\documentclass[12pt,a4paper]{article}
\usepackage[cp1251]{inputenc}
\usepackage{amsmath,amsthm}
\usepackage{amsfonts,amstext,amssymb,verbatim,epsfig}
\usepackage{dsfont}
\usepackage{enumerate}
\usepackage{psfrag}
\usepackage{graphicx}
\usepackage{subfig}
\usepackage{color}
\usepackage{float}
\usepackage{subfig}
\usepackage{enumitem}
\usepackage{scrextend}
\usepackage{mathrsfs}

\usepackage[top=3cm, bottom=3.5cm, left=2.5cm, right=2.5cm]{geometry} 

\def\R{{\mathbb{R}}}
\def\N{{\mathbb{N}}}
\def\Z{{\mathbb{Z}}}
\def\P{{\mathbb{P}}}
\def\prw{{\mathrm{P}}}

\def \p{{\bf P1}} 
\def \pp{{\bf P2}} 
\def \ppp{{\bf P3}} 
\def \s{{\bf S1}} 
\def \sss{{\bf S2}}
\def \d{{\bf D}}

\def \GG{{\mathbb G}} 
\newcommand{\cube}{\mathrm{Q}} 
\newcommand{\edges}{\square}   
\newcommand{\nngg}{\,\stackrel{\GG_0}\sim\,} 
\newcommand{\capa}{\mathrm{cap}} 
\newcommand{\lt}{\mathrm{n}}
\newcommand{\ltloop}{\mathcal{N}}
\newcommand{\good}{\mathrm{G}}
\newcommand{\bad}{\mathrm{B}}
\newcommand{\nume}{\mathcal{Z}}

\def \atom{\omega} 
\def \set{{\mathcal S}} 
\def \constP{{\chi_{\scriptscriptstyle {\mathrm P}}}}
\def \constS{{\Delta_{\scriptscriptstyle {\mathrm S}}}}
\def \RP{R_{\scriptscriptstyle {\mathrm P}}}
\def \RS{R_{\scriptscriptstyle {\mathrm S}}}
\def \LP{L_{\scriptscriptstyle {\mathrm P}}}
\def \epsP{{\varepsilon_{\scriptscriptstyle {\mathrm P}}}}
\def\badseed{{\overline G}}

\def\bb{{\mathrm B}_{\GG_0}}
\def\ss{{\mathrm S}_{\GG_0}}
\def\dint{\partial_{\mathrm{int}}}
\def\dext{\partial_{\mathrm{ext}}}
\def \ball{{\mathrm B}} 

\newtheorem{theorem}{Theorem}[section]
\newtheorem{corollary}[theorem]{Corollary}
\newtheorem{lemma}[theorem]{Lemma}
\newtheorem{proposition}[theorem]{Proposition}

\newtheoremstyle{likedef}
  {}%
  {}%
  {}%
  {}
  {\bfseries}%
  {.}%
  {.5em}%
  {}%

\theoremstyle{likedef}

\newtheorem{definition}[theorem]{Definition}
\newtheorem{remark}[theorem]{Remark}

\numberwithin{equation}{section}

\begin{document}

\title{Decoupling inequalities and supercritical percolation for the vacant set of random walk loop soup}

\author{
Caio Alves\thanks{
University of Leipzig, Department of Mathematics,  
Augustusplatz 10, 04109 Leipzig, Germany.
Email: caio.alves@math.uni-leipzig.de and artem.sapozhnikov@math.uni-leipzig.de. 
The research of the authors has been supported by the DFG grant SA 3465/1-1.}
\and
Artem Sapozhnikov\footnotemark[1]
}

\maketitle

\footnotetext{MSC2000: Primary 60K35, 82B43, 60G55, 60J10.}
\footnotetext{Keywords: Random walk loop soup; percolation; decoupling inequality; long-range correlations; Poisson point process; random walk.}

\begin{abstract}
It has been recently understood \cite{DRS12,PRS,S14} that for a general class of percolation models on $\Z^d$ satisfying suitable 
decoupling inequalities, which includes i.a.\ Bernoulli percolation, random interlacements and level sets of the Gaussian free field, 
large scale geometry of the unique infinite cluster in strongly percolative regime is qualitatively the same; 
in particular, the random walk on the infinite cluster satisfies the quenched invariance principle, 
Gaussian heat-kernel bounds and local CLT. 

In this paper we consider the random walk loop soup on $\Z^d$ in dimensions $d\geq 3$. 
An interesting aspect of this model is that despite its similarity and connections to random interlacements and the Gaussian free field, 
it does not fall into the above mentioned general class of percolation models, since the required decoupling inequalities are not valid. 

We identify weaker (and more natural) decoupling inequalities and prove that (a) they do hold for the random walk loop soup and 
(b) all the results about the large scale geometry of the infinite percolation cluster proved for the above mentioned class 
of models hold also for models that satisfy the weaker decoupling inequalities. Particularly, all these results are new 
for the vacant set of the random walk loop soup. 
(The range of the random walk loop soup has been addressed by Chang \cite{Ch15} by a model specific approximation method, which does not apply to the vacant set.)

Finally, we prove that the strongly supercritical regime for the vacant set of the random walk loop soup is non-trivial. 
It is expected, but open at the moment, that the strongly supercritical regime coincides 
with the whole supercritical regime. 
\end{abstract}

\section{Introduction}

Consider the integer lattice $\Z^d$ with dimension $d\geq 3$. 
Any nearest neighbor path $\dot\ell = (x_1,\dots,x_n)$ on $\Z^d$ with $x_n$ 
being a neighbor of $x_1$ is called a (non-trivial discrete) \emph{based loop}. 
Two based loops of length $n$ are equivalent if they differ only by a circular permutation of their vertices, i.e., 
$(x_1,\dots,x_n)$ is equivalent to $(x_i,\dots,x_n,x_1,\dots,x_{i-1})$ for all $i$. 
Equivalence classes of based loops for this equivalence relation are called \emph{loops}. 
Consider the measure $\dot\mu$ on based loops defined by 
\[
\dot\mu(\dot\ell)= \frac{1}{n}\,\left(\frac{1}{2d}\right)^{n}
,\quad \dot\ell = (x_1,\ldots, x_n),
\]
and denote the push-forward of $\dot\mu$ on the space of loops by $\mu$. 
For $\alpha>0$, let $\mathscr L^\alpha$ be the Poisson point process of loops with intensity measure $\alpha\mu$
(random walk loop soup). 

\medskip

Poisson ensembles of Markovian loops (loop soups) have been recently actively researched by probabilists and mathematical physicists 
partly due to their connections to the Gaussian free field, the Schramm-Loewner Evolution and the loop erased random walk, 
see, e.g., \cite{LW04, LeJ11, SW12, Szn12:GFF-notes, Lupu-RI, Lupu-CLE, Cam15, BCL16, SS-LEW}. 
Although they already appear implicitly in the work of Symanzik \cite{Sym69} on representations of the $\phi^4$ Euclidean field, 
the first mathematically rigorous definitions were given by Lawler and Werner \cite{LW04} in the context of planar Brownian motion (Brownian loop soup)
and by Lawler and Trujillo Ferreras \cite{LT07} in discrete setting. 

Percolation of loop soups was first considered by Lawler and Werner \cite{LW04} and 
Sheffield and Werner \cite{SW12}, who identified, in particular, the value of the critical intensity 
for the planar Brownian loop soup. 
The existence of percolation phase transition for the random walk loop soup on $\Z^d$ 
and properties of the critical intensity have been investigated in \cite{LLem12, Lupu-RI, CS16, Lupu-ECP, Ch15}. 
Comprehensive analysis of connectivity properties of the random walk loop soup on $\Z^d$ in subcritical regime was achieved 
by Chang and the second author \cite{CS16} and in supercritical regime by Chang \cite{Ch15}. 

One of the main challenges for the study of connectivity properties of the loop soup is 
the polynomial decay of correlations (see \cite{CS16}). 
Models of percolation exhibiting strong spatial correlations have been of immense interest in the last decade, 
including the random interlacements, the vacant set of random interlacements and the level sets of the Gaussian free field, 
see, e.g., \cite{SznitmanAM,Sznitman:Decoupling,RSzn13}. 
Many of the methods (particularly, the coarse graining and Peierls-type arguments) developed for Bernoulli percolation do not apply to these models. 
The fundamental idea behind the major progress in understanding these models (which are monotone in their intensity parameters) is that 
the effect of correlations can be well dominated with a slight tilt of the intensity parameter (\emph{sprinkling}). 
This idea is formalized in correlation inequalities, known as \emph{decoupling inequalities} \cite{SznitmanAM,Sznitman:Decoupling,RSzn13,DRS-book,PR15,PT15,AP15,Rodriguez16}. 
A general class of percolation models, which satisfy a suitable decoupling inequality and contains the three models mentioned above, was considered in \cite{DRS12,PRS,S14}, 
where most of the geometric properties of the infinite percolation cluster, previously only known to hold for Bernoulli percolation, were proven. 
(See Section~\ref{sec:correlated-percolations} for a precise formulation of conditions from \cite{DRS12}.)
An interesting aspect of the random walk loop soup percolation is that it does not fall into this general class of models, since 
the decoupling inequalities assumed there (see condition \ppp{} in Section~\ref{sec:correlated-percolations}) are not valid. 
The main reason is that the error term in the decoupling inequality \ppp{} gets smaller on larger scales, 
while the stochastic behavior of macroscopic loops in the loop soup is scale invariant, see Remark~\ref{rem:p3-loopsoup} for some more details.

\medskip

The main goal of this paper is the study of geometric properties of connected components of 
the \emph{vacant set} of the loop soup $\mathscr L^\alpha$ 
--- the vertices of $\Z^d$ that do not belong to any of the loops in $\mathscr L^\alpha$ ---
which we denote by $\mathcal V^\alpha$. 
The vacant set exhibits a non-trivial percolation phase transition: there exists $\alpha_*\in(0,\infty)$ such that 
\begin{itemize}\itemsep0pt
\item
for $\alpha<\alpha_*$ there is almost surely a unique infinite connected component in $\mathcal V^\alpha$, 
\item
for $\alpha>\alpha_*$ all the connected components are almost surely finite. 
\end{itemize}
The fact that $\alpha_*<\infty$ is elementary, since $\mathcal V^\alpha$ is stochastically dominated by Bernoulli site percolation 
with parameter $\exp\left(-\frac{\alpha}{4d^2}\right)$ (by restricting $\mathscr L^\alpha$ to loops of length $2$), 
and the positivity of $\alpha_*$ follows from Theorem~\ref{thm:lu}. 
The uniqueness of the infinite cluster is not entirely trivial, since the so-called positive finite energy property 
fails for $\mathcal V^\alpha$, but still can be proved by a direct adaptation of the standard Burton-Keane argument \cite{BK89}, cf.\ 
Remark~\ref{rem:uniqueness}.

\medskip

Our main focus is on geometric properties of the unique infinite cluster of $\mathcal V^\alpha$. 
As already mentioned, a unified framework to study infinite clusters of (correlated) percolation models on $\Z^d$ was proposed in \cite{DRS12}, 
within which various results that were previously known only for supercritical Bernoulli percolation have been proven. 
These include i.a.\ quenched Gaussian heat kernel bounds, Harnack inequalities, invariance principle and local CLT for the simple random walk 
on the infinite cluster \cite{PRS,S14}.
The loop soup percolation does not fall into this general class of models, since decoupling inequalities \ppp{} assumed there are not valid, see Remark~\ref{rem:p3-loopsoup}.
However, Chang \cite{Ch15} was able to prove all the above mentioned results for the infinite cluster in the range of the loop soup $\mathscr L^\alpha$
by observing that the properties of the infinite cluster are predominantly determined by loops with bounded diameter. 
In a way, the infinite cluster is a small perturbation on top of the infinite cluster of truncated loops. 
His analysis relies substantially on the Poisson point process structure of the loop soup and cannot be adapted to 
the vacant set, which is thus considerably more difficult. 

\medskip

Our first result states that the range of $\mathscr L^\alpha$ does satisfy a decoupling inequality, 
which is however weaker than the one imposed in \cite{DRS12}, see Remarks~\ref{rem:d}(4) and \ref{rem:p3-loopsoup}. 
\begin{theorem}\label{thm:decoupling}(Decoupling inequalities) 
Let $\mathcal R^\alpha$ be the set of vertices visited by loops from $\mathscr L^\alpha$ (the range of $\mathscr L^\alpha$) and 
denote by $\mathbb E^\alpha$ the expectation with respect to the distribution of $\left\{\mathds{1}_{x\in\mathcal R^\alpha}\right\}_{x\in\Z^d}$ on $\{0,1\}^{\Z^d}$.
There exist constants $C,c$ such that for any $\alpha>0$, $\delta\in(0,1)$, integers $L,s\geq 1$, $x_1,x_2\in\Z^d$ with $\|x_1-x_2\|=sL$, 
and any functions $f_1,f_2:\{0,1\}^{\Z^d}\to [0,1]$ such that 
$f_i(\omega)$ only depends on values of $\omega_x$ with $\|x-x_i\|\leq L$,
\begin{enumerate}[leftmargin=*]
\item
if $f_2$ is increasing, then 
\begin{equation}\label{eq:decoupling}
\mathbb E^\alpha\left[f_1\,f_2\right] \leq 
\mathbb E^\alpha\left[f_1\right]\,\mathbb E^{\alpha+\delta}\left[f_2\right] + C\exp\left(\alpha-c\sqrt{\delta}s^{d-2}\right),
\end{equation}
\item
if $f_2$ is decreasing, then 
\begin{equation}\label{eq:decoupling:decreasing}
\mathbb E^\alpha\left[f_1\,f_2\right] \leq 
\mathbb E^\alpha\left[f_1\right]\,\mathbb E^{(\alpha-\delta)_+}\left[f_2\right] + C\exp\left(\alpha-c\sqrt{\delta}s^{d-2}\right).
\end{equation}
\end{enumerate}
\end{theorem}
It turns out that the decoupling inequalities of Theorem~\ref{thm:decoupling} are strong enough to obtain the same results about the infinite cluster of $\mathcal V^\alpha$ 
as those derived for the class of models from \cite{DRS12}. 
More precisely, in Section~\ref{sec:correlated-percolations}, after recalling the assumptions from \cite{DRS12}, 
we prove that condition \ppp{} on spatial correlations can be relaxed, cf.\ condition \d{} in Section~\ref{sec:correlated-percolations}, 
without any effect on the conclusions of \cite{DRS12} 
and of \cite{PRS,S14} where the framework of \cite{DRS12} was further used, see Theorem~\ref{thm:badseed:proba} and Corollary~\ref{cor:P1-S2:D}. 
Crucially, even though the vacant set $\mathcal V^\alpha$ does not satisfy condition \ppp{}, it does satisfy the weaker condition \d{} by Theorem~\ref{thm:decoupling} (see Remark~\ref{rem:d}(4)). 

Furthermore, let us emphasize that condition \d{} is not only weaker than \ppp{}, but also more natural, since it 
postulates decorrelation of local events occuring in large boxes only when the boxes are far apart.
All in all, we believe that Theorem~\ref{thm:badseed:proba} and Corollary~\ref{cor:P1-S2:D} are of independent importance beyond their application in the present paper, 
nevertheless, we postpone their formulation to Section~\ref{sec:correlated-percolations} because of a large amount of necessary notation.

Incidentally, the results of Chang \cite{Ch15} about the geometry of the infinite cluster in the range of the loop soup 
can now be directly deduced as a special case of Corollary~\ref{cor:P1-S2:D} (and Theorem~\ref{thm:decoupling}). 
\begin{remark}
It is natural to ask if the error term of decoupling inequalities \eqref{eq:decoupling} and \eqref{eq:decoupling:decreasing} is optimal. 
We believe it is not, but do not know a good heuristics. 
Our proof is based on a delicate interplay between probabilities of two rare events (excess in the number of large loop excursions near $x_1$, resp., $x_2$) 
and it looks so that our result is optimal for the method, see Remark~\ref{rem:decoupling:proof}. 
For the application of Theorem~\ref{thm:decoupling} in this paper (Theorem~\ref{thm:geometry-infinite-cluster}), 
an error term in the form $C\exp\left(-c\,\delta^\beta\,s^\gamma\right)$ with some $\beta,\gamma>0$ would suffice, see Corollary~\ref{cor:P1-S2:D} and Remark~\ref{rem:DRS-conditions-mathcalV}.
\end{remark}

\medskip

Our next result proves that for small enough values of $\alpha$, the vacant set $\mathcal V^\alpha$ contains 
with high probabilitity a unique giant cluster in all large enough boxes. 
In particular, it implies that the supercritical phase is non-trivial ($\alpha_*>0$).

\begin{theorem}\label{thm:lu}(Local uniqueness)
For any $d\geq 3$ there exist $\alpha_1>0$, $c=c(d)>0$ and $C=C(d)<\infty$ such that for all $0\leq\alpha\leq\alpha_1$ and $n\geq 1$, 
\begin{equation}\label{eq:lu:1}
\P\left[\begin{array}{c}\text{the infinite connected component of $\mathcal V^\alpha$}\\ \text{intersects $\ball(0,n)$}\end{array}\right] \geq 1 - Ce^{-n^c}
\end{equation}
and 
\begin{equation}\label{eq:lu:2}
\P\left[\begin{array}{c}\text{any two connected subsets of $\mathcal V^\alpha\cap \ball(0,n)$ with}\\ 
\text{diameter $\geq \frac{n}{10}$ are connected in $\mathcal V^\alpha\cap \ball(0,2n)$}\end{array}\right] \geq 1 - Ce^{-n^c}.
\end{equation}
\end{theorem}
Properties \eqref{eq:lu:1} and \eqref{eq:lu:2} appear as assumption \s{} in the framework of \cite{DRS12}, see Section~\ref{sec:correlated-percolations}. 
The remaining conditions (ergodicity, monotonicity, continuity) from \cite{DRS12} are easily verified for $\mathcal V^\alpha$, see Remark~\ref{rem:DRS-conditions-mathcalV}. 
As a result, we can summarize the main conclusions about the geometry of the infinite cluster of $\mathcal V^\alpha$ as follows. 
(This is an immediate application of Theorem~\ref{thm:decoupling}, Corollary~\ref{cor:P1-S2:D} and Remark~\ref{rem:DRS-conditions-mathcalV}.)
\begin{theorem}\label{thm:geometry-infinite-cluster}
Let $d\geq 3$ and $\alpha_1>0$. If \eqref{eq:lu:1} and \eqref{eq:lu:2} hold 
for all $\alpha<\alpha_1$ with constants $c=c(d,\alpha)>0$ and $C=C(d,\alpha)<\infty$, then 
the unique infinite cluster of $\mathcal V^\alpha$ satisfies all the results from \cite{DRS12,PRS,S14} for all $\alpha<\alpha_1$, more precisely, 
\begin{itemize}\itemsep0pt
\item
Theorems~2.3 (chemical distances) and 2.5 (shape theorem) in \cite{DRS12},
\item
Theorem~1.1 in \cite{PRS} (quenched invariance principle),
\item
Theorem~1.13 (Barlow's ball regularity), Corollary~1.14 (quenched Gaussian heat kernel bounds, elliptic and parabolic Harnack inequalities), 
Theorem~1.19 (quenched local CLT), as well as Theorems~1.16--1.18, 1.20 in \cite{S14}.
\end{itemize}
We refer the reader to the introduction of \cite{S14} for the precise statements of these results.
\end{theorem}
We strongly believe that properties \eqref{eq:lu:1} and \eqref{eq:lu:2} with some $c=c(d,\alpha)>0$ and $C=C(d,\alpha)<\infty$ hold 
for all $\alpha<\alpha_*$. 
This has been proven to hold for Bernoulli percolation (for all $p>p_c$, see \cite[(7.89)]{Grimmett}), 
the random interlacements (for all $u>0$, see \cite{RS:Transience}) 
and for the range of the loop soup (for all $\alpha>\alpha_c$, see \cite{Ch15}), 
but is still conjectured for the level sets of the Gaussian free field and for the vacant set of random interlacements. 
(Analogues of Theorem~\ref{thm:lu} are proved for the level sets of the Gaussian free field on $\Z^d$ in \cite{DRS12} 
and on transient graphs from a broad class in \cite{DPR18} and for the vacant set of random interlacements on $\Z^d$ in \cite{Teixeira} (for $d\geq 5$) and \cite{DRS-AIHP} (for $d\geq 3$).)

\medskip

\emph{Overview of the paper.} In Section~\ref{sec:notation-preliminaries} we collect basic definitions and classical results on random walks. 
In Section~\ref{sec:loop-excursions} we study the Poisson point process of loops that intersect two disjoint sets. Such loops can be cut into successive excursions 
between the two sets which are distributed as independent random walk bridges conditioned on their starting and ending points, see Proposition~\ref{prop:sampling-loopsoup}. 
In Section~\ref{sec:proof-decoupling} we prove Theorem~\ref{thm:decoupling} and in Section~\ref{sec:proof-local-uniqueness} Theorem~\ref{thm:lu}. 
Finally in Section~\ref{sec:correlated-percolations}, which can be read independently of all the other sections, 
we recall the general conditions on percolation models from \cite{DRS12}, formulate a weaker 
decoupling inequality \d{} and prove in Theorem~\ref{thm:badseed:proba} that the condition \ppp{} from \cite{DRS12} can be substituted by \d{} 
without any loss in conclusions. The punchline of Section~\ref{sec:correlated-percolations} is Corollary~\ref{cor:P1-S2:D}, 
which particularly gives Theorem~\ref{thm:geometry-infinite-cluster}.

\section{Notation and preliminaries}\label{sec:notation-preliminaries}

For $x\in\Z^d$, let $\|x\|$ and $\|x\|_1$ be the $\ell_\infty$-, resp., $\ell_1$-norm of $x$ and denote by $\ball(x,r)$ the $\ell_\infty$ closed ball in $\Z^d$ of radius $r$ centered in $x$. 

\smallskip

For a set $A\subseteq\Z^d$, let $\dint A = \{y\in A:\|y'-y\|_1 = 1\text{ for some }y'\in \Z^d\setminus A\}$ be the interior boundary of $A$ 
and $\dext A = \{y\notin A:\|y'-y\|_1 = 1\text{ for some }y'\in A\}$ the exterior boundary of $A$.

\smallskip

A function $f:\{0,1\}^{\Z^d}\to\R$ is called increasing if $f(\omega)\leq f(\omega')$ for any $\omega,\omega'\in\{0,1\}^{\Z^d}$ such that  
$\omega_x\leq \omega_x'$ for all $x\in\Z^d$. A subset $E$ of $\{0,1\}^{\Z^d}$ is called increasing if its indicator $\mathds{1}_E$ is increasing 
($\mathds{1}_E(\omega) = 1$ if $\omega\in E$ and $0$ otherwise). A function $f$, resp., a set $E$, is called decreasing if $-f$, resp., $\{0,1\}^{\Z^d}\setminus E$, is increasing.

\smallskip

Let $W_+$ be the set of all infinite nearest neighbor paths on $\Z^d$ endowed with the $\sigma$-algebra generated by coordinate maps $X_n$, $n\in \N$.  
Denote by $\prw_x$ the law of a simple random walk on $\Z^d$ started at $x$ and by 
$g:\Z^d\times\Z^d\to\R$ the Green function of the simple random walk, 
$g(x,y) = \sum_{n=0}^\infty \prw_x[X_n = y]$. 
It is well known, see, e.g., \cite[Theorem~1.5.4]{LawlerRW}, that for any $d\geq 3$, there exist $c_g>0$ and $C_g<\infty$ such that 
\begin{equation}\label{eq:GF}
 c_g \, (\|x-y\|+1)^{2-d} \leq  g(x,y) \leq C_g \, (\|x-y\|+1)^{2-d}, \quad x,y\in\Z^d.
\end{equation}

\smallskip

For $A\subset \Z^d$ and a nearest neighbor path $w=(w_0,\ldots,w_N)$ on $\Z^d$, where $N\in\N_0\cup\{+\infty\}$, 
let $H_A(w) = \inf\{n \geq 0 \,:\, w_n \in A\}$ be the entrance time in $A$ 
and $\widetilde H_A(w) = \inf\{ n \geq 1 \, :\, w_n \in A\}$ the hitting time of $A$. 
The equilibrium measure of a finite set $A$ is defined by $e_A(x) = \prw_x[\widetilde H_A = \infty]\mathds{1}_{A}(x)$. 
Its total mass is the capacity of $A$, $\capa(A) = \sum_x e_A(x)$. The equilibrium measure of any finite set in dimensions $d\geq 3$ is non-zero 
and we denote by $\widetilde e_A$ the normalized equilibrium measure. 
The following relation between the entrance time probability, the Green function and the equilibrium measure is classical, see, e.g., \cite[(1.8)]{SznitmanAM}: 
\begin{equation}\label{eq:H-g-e}
P_x[H_A<\infty] = \sum_{y\in A} g(x,y) e_A(y) .
\end{equation}
By taking $x=0$ and $A=\dint \ball(0,n)$ in \eqref{eq:H-g-e} and using \eqref{eq:GF}, one easily gets the bounds on the capacity of balls:
\begin{equation}\label{eq:capa:ball}
c_c\,n^{d-2}\leq \capa\left(\ball(0,n)\right) = \capa\left(\dint \ball(0,n)\right)\leq C_c \, n^{d-2}.
\end{equation}
The following lemma and corollary are also standard. They will be used in the proof of Theorem~\ref{thm:decoupling}. 
\begin{lemma}\label{l:hitting}
There exist constants $c=c(d)>0$ and $C=C(d)<\infty$ such that 
\begin{enumerate}[leftmargin=*]
\item 
for all $n\geq 1$ and $x\notin \ball(0,n)$, 
\begin{equation}\label{eq:hitting1}
c\,\left(\frac{n}{\|x\|}\right)^{d-2}\leq \prw_x\left[H_{\ball(0,n)}<\infty\right] \leq C\,\left(\frac{n}{\|x\|}\right)^{d-2},
\end{equation}
\item
for all $n\geq 1$, $m>2n$, $A\subset \ball(0,n)$, $x\notin \ball(0,m)$ and $y\in A$, 
\begin{equation}\label{eq:hitting2}
c\,\widetilde e_A(y)\leq \prw_x\left[X_{H_A} = y~|~H_A<\infty\right] \leq C\,\widetilde e_A(y).
\end{equation}
\end{enumerate}
\end{lemma}
\begin{proof}
The first statement is immediate from \eqref{eq:GF}, \eqref{eq:H-g-e} and \eqref{eq:capa:ball}.
The second follows from \cite[Theorem~2.1.3]{LawlerRW} and the Harnack principle (see, e.g., \cite[Theorem~1.7.6]{LawlerRW}). 
\end{proof}
\begin{corollary}
Let $L\geq 1$, $2< r\leq \frac12 s$ be integers, $x_1,x_2\in\Z^d$ with $\|x_1-x_2\|=sL$, 
and define $S_i = \dint\ball(x_i,L)$ and $S_i' = \dint\ball(x_i, rL)$, $i\in\{1,2\}$. 

There exist constants $c=c(d)>0$ and $C=C(d)<\infty$ such that for all $r>C$, $x\in S_1'$ and $y\in S_2$, 
\begin{equation}\label{eq:hitting3}
c\,\prw_x\left[H_{S_2}<\infty\right]\,\widetilde e_{S_2}(y)
\leq \prw_x\left[H_{S_2}<H_{S_1}, X_{H_{S_2}} = y\right] \leq C\,\prw_x\left[H_{S_2}<\infty\right]\,\widetilde e_{S_2}(y).
\end{equation}
\end{corollary}
\begin{proof}
Immediate from Lemma~\ref{l:hitting} and the Markov property of random walk.
\end{proof}

\smallskip

For $A\subset\Z^d$, $x\notin A$, $y\in A$, consider the law 
\[
\prw_{x,y}^A = \prw_x\left[(X_0,\ldots,X_{H_A})=\cdot~|~X_{H_A} = y\right]
\]
of a random walk path (bridge) from $x$ conditioned to enter $A$ at $y$.

\bigskip

The set of all based loops is denoted by $\dot{\mathfrak L}$ and all loops by $\mathfrak L$. 
For a loop $\ell\in\mathfrak L$ and $A\subset\Z^d$, we write $\ell\cap A\neq\emptyset$ if 
some (and hence all) representative from the equivalence class $\ell$ contains at least one vertex in $A$. 
If $A=\{x\}$, then we instead write $x\in\ell$. If $\mathcal L$ is a subset of $\mathfrak L$ and $x\in\Z^d$, then we write $x\in\mathcal L$ 
if there exists $\ell\in\mathcal L$ such that $x\in\ell$.

We denote by $\pi:\dot{\mathfrak L}\to\mathfrak L$ the canonical projection, i.e., $\pi(\dot\ell)$ is the equivalence class of $\dot\ell$.
Consider the measure $\dot\mu$ on $\dot{\mathfrak L}$ defined by 
\begin{equation}\label{def:dotmu}
\dot\mu(\dot\ell)= \frac{1}{n}\,\prw_{x_1}\left[(X_0,\ldots, X_{n-1}) = \dot\ell,\, X_n = x_1\right] 
=\frac{1}{n}\,\left(\frac{1}{2d}\right)^{n}
,\quad \dot\ell = (x_1,\ldots, x_n),
\end{equation}
and denote by $\mu$ the push-forward of $\dot\mu$ on $\mathfrak L$ by $\pi$.

For $\alpha>0$ let 
\begin{itemize}\itemsep0pt
\item
$\mathscr L^\alpha$ be the Poisson point process of loops with intensity measure $\alpha\mu$,
\item
$\mathcal N^\alpha$ the field of cumulative local times for the loops in $\mathscr L^\alpha$,
\item
$\mathcal V^\alpha = \{x\in\Z^d:\mathcal N^\alpha(x) = 0\}$ the vacant set for $\mathscr L^\alpha$.
\end{itemize}
We assume that these processes are defined on a probability space $(K,\mathcal K,\P)$, 
whose precise description is irrelevant and also use $\P^\alpha$ and $\mathbb E^\alpha$ to denote the law, resp., expectation, 
of $\{\mathds{1}_{x\in\mathscr L^\alpha}\}_{x\in\Z^d}$ on $\{0,1\}^{\Z^d}$. 

\medskip

Constants that only depend on the dimension (and in Seciton~\ref{sec:correlated-percolations} possibly also on $a$ and $b$) 
are denoted by $c$ and $C$. Their value may change from line to line and even within lines.

\section{Decomposition of loops in excursions}\label{sec:loop-excursions}

In this section we study properties of loops that visit two disjoint sets $A,B\subset\Z^d$. 
Any such loop can be cut into alternating excursions from $A$ to $B$ and from $B$ to $A$, 
which, given their starting and ending points, are distributed as independent random walk bridges. 
This gives a useful way to sample the Poisson point process of loops that visit $A$ and $B$, 
see Proposition~\ref{prop:sampling-loopsoup}. Furthermore, the total number of loop excursions 
is unlikely to be large if $A$ and $B$ are far apart, see Lemma~\ref{l:excursions-AB-num}. 

\smallskip

Let $A,B\subset\Z^d$ be disjoint and consider the set of all loops that visit $A$ and $B$:
\begin{equation*}
\mathfrak{L}_{A,B}:=\left\{\ell\in\mathfrak{L}: \ell\cap A\neq\emptyset,\ell\cap B \neq \emptyset \right\}.
\end{equation*}
We first recall a useful representation of the measure $\mu$ on $\mathfrak L_{A,B}$ from \cite{CS16}.

\begin{definition}
For each $\ell\in \mathfrak L$, let $L(A,B)(\ell)$ be the set of all based loops $\dot\ell = (x_1,\ldots, x_n)$ from the 
equivalence class $\ell$ such that 
\begin{itemize}\itemsep0pt
\item
$x_1\in A$, 
\item
there exists $i$ such that $x_i\in B$ and for all $j>i$ (if exists) $x_j\notin (A\cup B)$. 
\end{itemize}
\end{definition}
Note that 
\begin{itemize}\itemsep0pt
\item
$L(A,B)(\ell)\cap L(A,B)(\ell') = \emptyset$ if $\ell\neq \ell'$,
\item
$L(A,B)(\ell)\neq\emptyset$ if and only if $\ell\in\mathfrak L_{A,B}$. 
\end{itemize}
Any loop in $\mathfrak L_{A,B}$ can be decomposed into alternating nearest neighbor excursions from $A$ to $B$ and from $B$ to $A$. 
For any $\ell\in \mathfrak L_{A,B}$ and $\dot\ell = (x_1,\dots,x_n)\in L(A,B)(\ell)$, we define the entrance times 
\begin{equation}\label{eq:excursions-times}
\begin{split}
\phi_1(\dot\ell) &= 1,\\ 
\psi_1(\dot\ell) & = \inf\left\{j>\phi_1(\dot\ell)~:~x_j\in B\right\},\\
\phi_k(\dot\ell) & = \inf\left\{j>\psi_{k-1}(\dot\ell)~:~x_j\in A\right\},\\ 
\psi_k(\dot\ell) & = \inf\left\{j>\phi_k(\dot\ell)~:~x_j\in B\right\}, \quad k\geq 1,
\end{split}
\end{equation}
with $\inf\{\emptyset\} = \infty$, and let 
\[
k(\dot\ell) = \sup\{n\geq 1~:~\phi_n(\dot\ell)<\infty\} < \infty.
\]
Note that the value of $k(\dot\ell)$ is the same for all $\dot\ell\in L(A,B)(\ell)$, in fact $k(\dot\ell) = |L(A,B)(\ell)|$, and we denote it by $k(\ell)$. 

\begin{lemma}\label{l:claim1}\cite[Claim~1]{CS16}
For any loop $\ell \in\mathfrak L_{A,B}$,  
\begin{eqnarray*}
\mu(\ell) &= &\frac{|\ell|}{k(\ell)}\, \sum_{\dot\ell\in L(A,B)(\ell)}\, \dot\mu(\dot\ell)\\
&\stackrel{\eqref{def:dotmu}}= &\frac{1}{k(\ell)}\, \sum_{x\in A}\,\prw_x\left[(X_0,\ldots, X_{|\ell|-1})\in L(A,B)(\ell),\, X_{|\ell|} = x\right],
\end{eqnarray*}
where $|\ell|$ is the length of the loop $\ell$. 
\end{lemma}

\bigskip

Let $\mathscr L^\alpha_{A,B}$ be the restriction of $\mathscr L^\alpha$ to $\mathfrak L_{A,B}$. 
It is a Poisson point process with intensity measure $\alpha\,\mathds{1}_{\mathfrak L_{A,B}}\,\mu$, 
which is independent from the restriction of $\mathscr L^\alpha$ to $\mathfrak L\setminus \mathfrak L_{A,B}$. 
We are interested in the distribution of excursions from $A$ to $B$ of the loops in $\mathscr L^\alpha_{A,B}$ 
(parts of the loop between times $\phi_i$ and $\psi_i$). The set of excursions is only determined up to cyclic permutations, 
therefore, it is more convenient to work with excursions of based loops. 
The following lemma identifies $\mathscr L^\alpha_{A,B}$ with a projection of a suitable Poisson point process of based loops. 
Let 
\[
\dot{\mathfrak L}_{A,B} = L(A,B)(\mathfrak L_{A,B}),
\]
i.e., the set of all based loops $\dot\ell = (x_1,\ldots, x_n)$ such that 
\begin{itemize}\itemsep0pt
\item
$x_1\in A$, 
\item
there exists $i$ such that $x_i\in B$ and $x_j\notin (A\cup B)$ for all $j>i$. 
\end{itemize}
(Mind that $\dot{\mathfrak L}_{A,B}$ is \emph{not} the set of all based loops that intersect $A$ and $B$, as may be suggested by notation.)

\begin{lemma}\label{l:LalphaAB-basedloops}
Let $\dot{\mathscr L}^\alpha_{A,B}$ be a Poisson point process on $\dot{\mathfrak L}_{A,B}$ with intensity measure $\alpha\dot\mu_{A,B}$, where 
\[
\dot\mu_{A,B}(\dot\ell) = 
\frac{1}{k(\dot{\ell})}\,\prw_{x_1}\left[(X_0,\ldots, X_{|\dot\ell|-1})=\dot\ell,\, X_{|\dot\ell|} = x_1\right],\qquad 
\dot\ell = (x_1,\ldots, x_{|\dot\ell|})\in\dot{\mathfrak L}_{A,B}\,.
\]
Then $\pi\left(\dot{\mathscr L}^\alpha_{A,B}\right)$ is a Poisson point process on $\mathfrak L_{A,B}$ with intensity measure $\alpha\,\mathds{1}_{\mathfrak L_{A,B}}\,\mu$.

In other words, to sample $\mathscr L^\alpha_{A,B}$ one first samples the Poisson point process $\dot{\mathscr L}^\alpha_{A,B}$ of based loops 
and then replaces each based loop by its equivalence class. 
\end{lemma}
\begin{proof}
This is a direct consequence of Lemma~\ref{l:claim1}. Indeed, $\pi\left(\dot{\mathscr L}^\alpha_{A,B}\right)$ is a Poisson point process with intensity measure 
\[
\ell\,\mapsto\, \alpha\,\sum_{\dot\ell\in L(A,B)(\ell)}\,\dot\mu_{A,B}(\dot\ell) = \alpha\mu(\ell)\,,\qquad \ell\in \mathfrak L_{A,B}\,.
\]
\end{proof}
The advantage of based loops in $\dot{\mathscr L}^\alpha_{A,B}$ is that their excursions from $A$ to $B$ are naturally ordered. 
Of course, the range of all based loops in $\dot{\mathscr L}^\alpha_{A,B}$ has the same law as the range of all loops in $\mathscr L^\alpha_{A,B}$.

\medskip

Next, we decompose the Poisson point process $\dot{\mathscr L}^\alpha_{A,B}$ according to the number of excursions that a based loop makes from $A$ to $B$. 
Namely, for $j\geq 1$, we denote by $\dot{\mathscr L}^{\alpha,j}_{A,B}$ the restriction of $\dot{\mathscr L}^\alpha_{A,B}$ 
to $\dot{\mathfrak L}^j_{A,B}=\{\dot\ell\in\dot{\mathfrak L}_{A,B}~:~ k(\dot\ell) = j\}$. Then, 
\begin{itemize}\itemsep0pt
\item
$\dot{\mathscr L}^{\alpha,j}_{A,B}$, $j\geq 1$, are independent Poisson point processes, 
\item
the intensity measure of $\dot{\mathscr L}^{\alpha,j}_{A,B}$ is $\alpha\mathds{1}_{\dot{\mathfrak L}^j_{A,B}}\dot\mu_{A,B}$,
\item
$\dot{\mathscr L}^{\alpha}_{A,B} = \sum_{j=1}^\infty\,\dot{\mathscr L}^{\alpha,j}_{A,B}$.
\end{itemize}

\medskip

We show in Proposition~\ref{prop:sampling-loopsoup} that each loop soup $\dot{\mathscr L}^{\alpha,j}_{A,B}$ can be constructed 
by sampling the starting and ending locations of all the excursions from $A$ to $B$ of all the loops in $\dot{\mathscr L}^{\alpha,j}_{A,B}$ 
according to a Poisson point process 
and then joining the endpoints by independent random walk bridges.

Let $j\geq 1$ and recall $\phi_i$ and $\psi_i$ defined in \eqref{eq:excursions-times}. 
For a loop $\dot\ell = (x_1,\ldots, x_n)\in \dot{\mathfrak L}^j_{A,B}$, denote the starting and ending locations of 
all the excursions of $\dot\ell$ from $A$ to $B$ by 
\[
\Phi_i(\dot\ell) = x_{\phi_i(\dot\ell)}\,\in A,\quad \Psi_i(\dot\ell) = x_{\psi_i(\dot\ell)}\,\in B,\qquad 1\leq i\leq j,
\]
the excursions from $A$ to $B$ by 
\[
\overrightarrow W_i(\dot\ell) = \left(x_{\phi_i(\dot\ell)},\ldots, x_{\psi_i(\dot\ell)}\right),\qquad 1\leq i\leq j,
\]
and the excursions from $B$ to $A$ by 
\[
\begin{split}
\overleftarrow W_i(\dot\ell) &= \left(x_{\psi_i(\dot\ell)},\ldots, x_{\phi_{i+1}(\dot\ell)}\right),\qquad 1\leq i\leq j-1,\\ 
\overleftarrow W_j(\dot\ell) &= \left(x_{\psi_j(\dot\ell)},\ldots, x_n, x_1\right),
\end{split}
\]
see Figure~\ref{fig:ABexcursions} for an illustration; 
and consider the Poisson point processes (multisets)
\begin{figure}[!tp]
\centering
\resizebox{12cm}{!}{\includegraphics{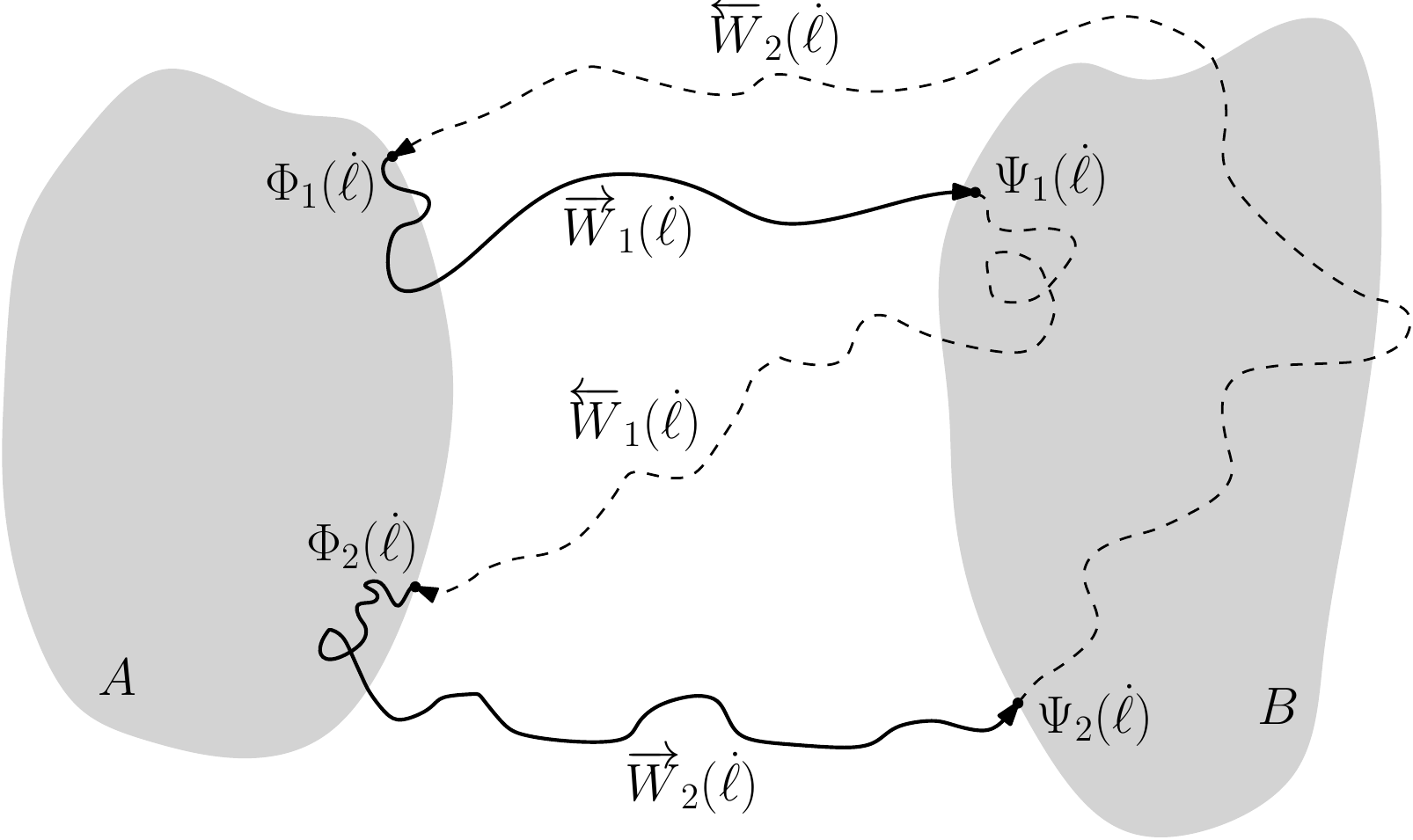}}
\caption{Decomposition of a loop from $\dot{\mathfrak L}^2_{A,B}$ into successive excursions.}
\label{fig:ABexcursions}
\end{figure}
\begin{equation}\label{def:mathcalE}
\begin{array}{rcl}
\mathcal E^{\alpha,j}_{A,B} &= &\left\{\left((\Phi_1(\dot\ell),\Psi_1(\dot\ell)),\ldots, (\Phi_j(\dot\ell),\Psi_j(\dot\ell))\right),\,\dot\ell\in \dot{\mathscr L}^{\alpha,j}_{A,B}\right\},\\[7pt]
\overrightarrow{\mathcal E}^{\alpha,j}_{A,B} &= &\left\{\left(\overrightarrow W_1(\dot\ell),\ldots, \overrightarrow W_j(\dot\ell)\right),\,\dot\ell\in \dot{\mathscr L}^{\alpha,j}_{A,B}\right\},\\[7pt]
\overleftarrow{\mathcal E}^{\alpha,j}_{A,B} &= &\left\{\left(\overleftarrow W_1(\dot\ell),\ldots, \overleftarrow W_j(\dot\ell)\right),\,\dot\ell\in \dot{\mathscr L}^{\alpha,j}_{A,B}\right\}.
\end{array}
\end{equation}

\begin{proposition}\label{prop:sampling-loopsoup}
Let $A,B$ be disjoint subsets of $\Z^d$, $d\geq 3$. 
For an infinite path $w=(x_0,x_1,\ldots)$, consider the sequence of times
\begin{equation*}
\begin{split}
\tau_0(w) & =  0,\\
\tau_{2j+1}(w) & = \inf\{ k>\tau_{2j}(w):x_k\in B  \},\\ 
\tau_{2j+2}(w) & = \inf\{ k>\tau_{2j+1}(w):x_k\in A  \},\qquad j\geq 0\,,
\end{split}
\end{equation*}
where $\inf\emptyset=\infty$. 

Then, for any $\alpha>0$ and integer $j\geq 1$, 
\begin{enumerate}[leftmargin=*]
\item
the intensity of $\mathcal E^{\alpha,j}_{A,B}$ is 
\[
\left((a_1,b_1),\ldots, (a_j,b_j)\right)\in (A\times B)^j\,\longmapsto\,
\frac{\alpha}{j}\,\prw_{a_1}\left[
\begin{array}{l}
\tau_{2j}<\infty,\,X_{\tau_{2j}} = a_1,\\
X_{\tau_{2(i-1)}} = a_i,\, X_{\tau_{2i-1}} = b_i,\,1\leq i\leq j
\end{array}\right],
\]
\item
conditioned on the multiset $\mathcal E^{\alpha,j}_{A,B} = \left\{(a_{i1},b_{i1}),\ldots, (a_{ij},b_{ij}), 1\leq i\leq n\right\}$, 
the Poisson point processes $\overrightarrow{\mathcal E}^{\alpha,j}_{A,B}$ and $\overleftarrow{\mathcal E}^{\alpha,j}_{A,B}$ 
are independent and sampled as products of bridge measures $\prw_{a_{ik},b_{ik}}^B$, resp., $\prw_{b_{ik},a_{i(k+1)}}^A$,
\end{enumerate}
Thus, the loops from $\dot{\mathscr L}^{\alpha}_{A,B}$ can be sampled in steps: first sample the number and the starting and ending locations of all excursions of all loops in $\dot{\mathscr L}^\alpha_{A,B}$ by 
sampling independently $\mathcal E^{\alpha,j}_{A,B}$, $j\geq 1$, and then complete all the excursions by sampling 
independent random walk bridges from $\prw_{\cdot,\cdot}^B$, resp., $\prw_{\cdot,\cdot}^A$.
\end{proposition}
\begin{proof}
Let $j\geq 1$ and $\dot\ell = (x_1,\ldots, x_{|\dot\ell|})\in \dot{\mathfrak L}^j_{A,B}$. 
The result is immediate from the following representation of $\dot\mu_{A,B}$:
\begin{eqnarray*}
\dot\mu_{A,B}(\dot\ell)
&= &\frac{1}{j}\,\prw_{x_1}\left[(X_0,\ldots, X_{|\dot\ell|-1})=\dot\ell,\, X_{|\dot\ell|} = x_1\right]\\[4pt]
&= &\frac1j\,\prw_{x_1}\left[
\begin{array}{ll}
\tau_{2j}<\infty,\,X_{\tau_{2j}} = x_1, &
X_{\tau_{2(i-1)}} = \Phi_i(\dot\ell),\, X_{\tau_{2i-1}} = \Psi_i(\dot\ell),\\
& (X_{\tau_{2(i-1)}},\ldots, X_{\tau_{2i-1}})=\overrightarrow W_i(\dot\ell),\\
& (X_{\tau_{2i-1}},\ldots, X_{\tau_{2i}})=\overleftarrow W_i(\dot\ell),\,1\leq i\leq j
\end{array}\right]\\[4pt]
&= &\frac1j\,\prw_{x_1}\left[
\begin{array}{l}
\tau_{2j}<\infty,\,X_{\tau_{2j}} = x_1,\\
X_{\tau_{2(i-1)}} = \Phi_i(\dot\ell),\, X_{\tau_{2i-1}} = \Psi_i(\dot\ell),\,1\leq i\leq j
\end{array}\right]\\
& &\qquad\qquad\prod_{i=1}^j\,\prw_{\Phi_i(\dot\ell),\, \Psi_i(\dot\ell)}^B\left[\overrightarrow W_i(\dot\ell)\right]\,
\prod_{i=1}^j\,\prw_{\Psi_i(\dot\ell),\, \Phi_{i+1}(\dot\ell)}^A\left[\overleftarrow W_i(\dot\ell)\right]\,,
\end{eqnarray*}
where in the last step we used the Markov property of random walk and set $\Phi_{j+1} = \Phi_1$. 
\end{proof}

\begin{remark}\label{rem:uniqueness}
Proposition~\ref{prop:sampling-loopsoup} (applied to $A=\dint\ball(0,n)$, $B=\dext\ball(0,n)$) can be used 
to adapt to $\mathcal V^\alpha$ the standard Burton-Keane argument \cite{BK89} for the uniqueness of the infinite percolation cluster, 
even though one of the main requirements, the positive finite energy property, is not satisfied by $\mathcal V^\alpha$. 
(The positive finite energy property states that $\mathbb P\left[0\in\mathcal V^\alpha~|~\sigma\left(\mathds{1}_{x\in\mathcal V^\alpha},\,x\neq 0\right)\right]>0$ almost surely, 
which is obviously not the case here, since, for instance, if all the vertices of $\ball(0,2)\setminus\{0\}$ are vacant except for one neighbor of the origin, then 
the origin cannot be vacant, as every loop visits at least $2$ vertices.) See, e.g., \cite[Theorem~1.1]{Teixeira:uniqueness}, where the Burton-Keane argument is adapted to prove the uniqueness of the infinite percolation cluster 
in the vacant set of random interlacements, which also does not satisfy the positive finite energy property. 
\end{remark}

\medskip

We end this section with a large deviation bound on the total number of excursions from $A$ to $B$ in all loops from $\mathscr L^\alpha_{A,B}$. 
\begin{lemma}\label{l:excursions-AB-num}
Let $A,B$ be (disjoint) subsets of $\Z^d$ such that 
\begin{equation}\label{eq:excursions-AB-num-assumption}
\sup_{y\in B}\prw_y[H_A<\infty] \leq \frac{1}{2e}.
\end{equation}
Let ${\mathcal Z}^\alpha_{A,B}$ be the total number of excursions from $A$ to $B$ of all the loops from $\mathscr L^\alpha$. 
Then, 
\[
\P[{\mathcal Z}^\alpha_{A,B}\geq k]\leq \exp\left(\alpha - k\right).
\]
\end{lemma}
\begin{proof}
Let ${\mathcal Z}^{\alpha,j}_{A,B}$ be the number of loops in $\dot{\mathscr L}^{\alpha,j}_{A,B}$. 
By Proposition~\ref{prop:sampling-loopsoup}(1), ${\mathcal Z}^{\alpha,j}_{A,B}$ are independent Poisson random variables with intensities 
\begin{eqnarray}\label{eq:intensity-Zj}
\lambda_j &= &\frac{\alpha}{j} \sum_{x\in A}\prw_x\left[\tau_{2j}<\infty,\,X_{\tau_{2j}} = x\right]
\leq \frac{\alpha}{j} \sup_{z\in A}\prw_z\left[\tau_{2j}<\infty\right]
\stackrel{(*)}\leq \frac{\alpha}{j} \left(\sup_{z\in A}\prw_z\left[\tau_2<\infty\right]\right)^j \nonumber \\ 
&\stackrel{(**)}\leq &\frac{\alpha}{j} \left(\sup_{y\in B}\prw_y\left[H_A<\infty\right]\right)^j
\,\stackrel{\eqref{eq:excursions-AB-num-assumption}}
\leq\, \frac{\alpha}{j}\,\left(\frac{1}{2e}\right)^j,
\end{eqnarray}
where in ($*$) we used the strong Markov property at times $\tau_{2i}$, $1\leq i<j$, and in ($**$) at time $\tau_1$. 
Furthermore, by Lemma~\ref{l:LalphaAB-basedloops}, ${\mathcal Z}^{\alpha}_{A,B} \stackrel{d}= \sum_{j=1}^\infty j\, {\mathcal Z}^{\alpha,j}_{A,B}$. Thus,
\begin{eqnarray*}
\mathbb E\left[\exp\left({\mathcal Z}^{\alpha}_{A,B}\right)\right] 
&= &\mathbb E\left[\exp\left(\sum_{j=1}^\infty j\, {\mathcal Z}^{\alpha,j}_{A,B}\right)\right] = 
\prod_{j=1}^\infty\,\mathbb E\left[\exp\left(j\, {\mathcal Z}^{\alpha,j}_{A,B}\right)\right] \\
&= &\prod_{j=1}^\infty\, \exp\left(\lambda_j\left(e^j - 1\right)\right) 
\,\stackrel{\eqref{eq:intensity-Zj}}\leq\, e^{\alpha},
\end{eqnarray*}
and the result follows from the exponential Chebyshev inequality. 
\end{proof}

\section{Proof of Theorem~\ref{thm:decoupling}}\label{sec:proof-decoupling}

The proofs of \eqref{eq:decoupling} and \eqref{eq:decoupling:decreasing} are very similar and 
we only provide here the proof of \eqref{eq:decoupling}. 

We begin with an outline of the proof. 
We decompose the loops from $\mathscr L^\alpha$ that intersect $S_1 = \dint \ball(x_1,L)$ and $S_1' = \dint \ball(x_1,rL)$ (with a fixed large $r\in (2,s/2]$) 
into inner (from $S_1$ to $S_1'$) and outer (from $S_1'$ to $S_1$) excursions. 
By Proposition~\ref{prop:sampling-loopsoup}, given their starting and ending locations, the inner and outer excursions are independent random walk bridges. 
By the locality of $f_1$ and $f_2$ and disjointness of $\ball(x_1,rL)$ and $\ball(x_2,L)$, the inner excursions contribute only to the value of $f_1$ and the outer only to the value of $f_2$. 
By Lemma~\ref{l:excursions-AB-num} the total number of the outer excursions is bounded by $k$ with probability $\leq e^{\alpha-k}$.
For each outer excursion, its range in $\ball(x_2,L)$ is stochastically dominated by the range of a random walk loop soup with intensity $\frac{\delta}{k}$ 
on an event of probability $\geq 1 - \exp\left(\frac{\delta}{k}\,s^{2(d-2)}\right)$ (see Lemma~\ref{l:bridgerange}). 
Since $f_2$ is monotone and depends only on the configuration in $\ball(x_2,L)$, the stochastic domination implies the desired inequality for expectations.
Optimization over $k$ gives (for $k=\sqrt{\delta}s^{d-2}$) the desired error term.

\medskip

We proceed with the details of the proof. 
Without loss of generality, we may assume that $s\geq s_0=s_0(d)$. 
Let $L\geq 1$ and take $2 < r\leq s/2$ sufficiently large (the ultimate choice of $r$ depends only on the dimension). 
Let $x_1,x_2\in\Z^d$ with $\|x_1-x_2\|=sL$ and define 
\[
B_i = \ball(x_i,L),\quad B_i' = \ball(x_i,rL),\quad S_i = \dint B_i,\quad S_i' = \dint B_i'.
\]
Let $f_1,f_2:\{0,1\}^{\Z^d}\to[0,1]$ such that $f_i$ only depends on coordinates of $\omega\in\{0,1\}^{\Z^d}$ in $B_i$ and assume that $f_2$ is increasing. 

\smallskip

Let $\mathcal Z = \mathcal Z^\alpha_{S_1',S_1}$ be the total number of excursions from $S_1'$ to $S_1$ in $\mathscr L^\alpha$ and 
$\mathcal E = \{(\mathcal X_i,\mathcal Y_i):1\leq i\leq \mathcal Z\}$ the multiset of starting and ending positions of all the excursions 
(i.e., all the pairs from $\mathcal E^{\alpha,j}_{S_1',S_1}$, $j\geq 1$).
By Proposition~\ref{prop:sampling-loopsoup}, conditioned on $\mathcal E$, the excursions are distributed as independent random walk bridges 
started at $\mathcal X_i$ and conditioned to hit $S_1$ at $\mathcal Y_i$. 

Let $k\geq 1$ (to be specified later) and consider the event 
\[
G_1 = \{\mathcal Z\leq k\}.
\]
By the locality of $f_1$ and $f_2$, $f_1$ only depends on the loops from $\mathscr L^\alpha$ that are contained in $B_1'$ and 
on the excursions of the loops intersecting both $S_1$ and $S_1'$ that start on $S_1$ and end on $S_1'$, and 
$f_2$ is independent of all these loops and excursions given $\mathcal E$ by Proposition~\ref{prop:sampling-loopsoup} and the definition of Poisson point process. 
Thus, 
\begin{equation}\label{eq:G1-G1c}
\mathbb E^\alpha\left[f_1\, f_2\right] \leq 
\mathbb E^\alpha\left[f_1\mathds{1}_{G_1}\,\mathbb E^\alpha\left[f_2~|~\mathcal E\right]\right] + \P^\alpha[G_1^c].
\end{equation}
We will now bound $\mathbb E^\alpha\left[f_2~|~\mathcal E\right]$. 
For each $k$-tuple $\{(\widetilde x_i,\widetilde y_i)\in S_1'\times S_1,\, 1\leq i\leq k\}$, denote by 
$\mathbb E_{\alpha; \{(\widetilde x_i,\widetilde y_i)\}_{i=1}^k}$ the expectation with respect to the distribution of $\{\mathds{1}_{x\in \widetilde {\mathcal R}}\}_{x\in\Z^d}$ on $\{0,1\}^{\Z^d}$, 
where $\widetilde {\mathcal R}$ is the range of
\begin{itemize}\itemsep0pt
\item 
the loops from $\mathscr L^\alpha$ that do not intersect $S_1$ and 
\item
independent $k$-tuple of independent random walk bridges with the $i$th bridge starting at $\widetilde x_i$ and conditioned to hit $S_1$ at $\widetilde y_i$. 
\end{itemize}
By the monotonicity of $f_2$,  
\[
\mathbb E^\alpha\left[f_2~|~\mathcal E\right]\, \leq\, \max_{\left\{(\widetilde x_i,\widetilde y_i)\in S_1'\times S_1\right\}_{i=1}^k}\,
\mathbb E_{\alpha; \{(\widetilde x_i,\widetilde y_i)\}_{i=1}^k}\left[f_2\right] \quad\text{on $G_1$}.
\]

It now suffices to analyse separately the influence of each bridge on the configuration in $B_2$. 
We will prove the following lemma, which easily gives the main result.
\begin{lemma}\label{l:bridgerange}
For $x\in S_1'$ and $y\in S_1$, let $\mathcal R_{x,y}$ be the range in $B_2$ of a random walk bridge started at $x$ and conditioned to hit $S_1$ at $y$, see Figure~\ref{fig:range-Rxy}. 
For $\delta'\in(0,1)$, let $\mathcal R$ be the range in $B_2$ of the loops from the loop soup $\mathscr L^{\delta'}$. 

Then for each $r\geq r_0=r_0(d)$ there exists a coupling $(\overline{\mathcal R}_{x,y},\overline {\mathcal R})$ of $\mathcal R_{x,y}$ and $\mathcal R$ such that 
\[
\P\left[\overline{\mathcal R}_{x,y}\subseteq \overline {\mathcal R}\right]\geq 1 - Cs^{2(2-d)}\,\exp\left[-c\,\delta'\,s^{2(d-2)}\right],
\]
where $C=C(d,r)$ and $c=c(d,r)$. 
\end{lemma}
\begin{figure}[!tp]
\centering
\resizebox{12cm}{!}{\includegraphics{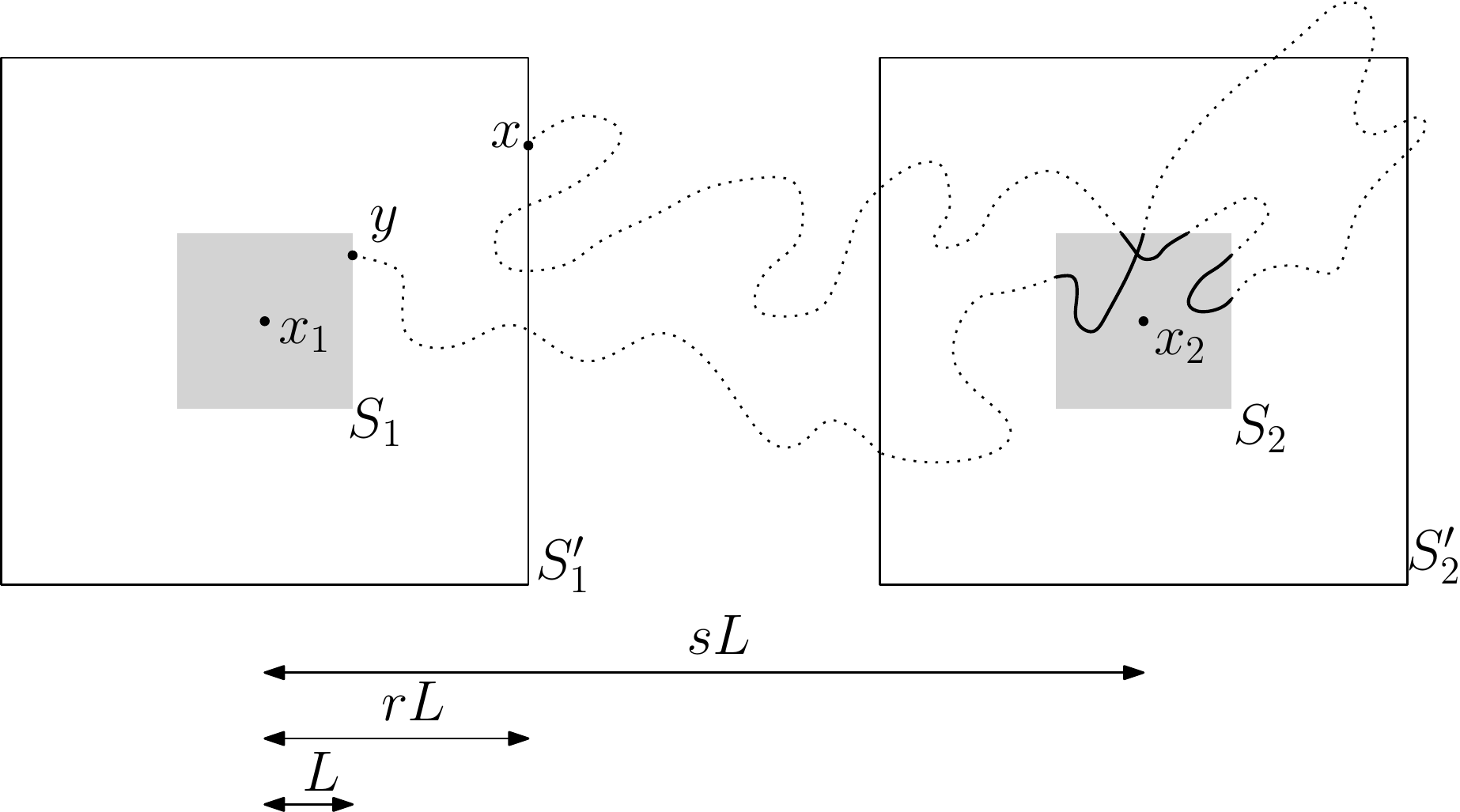}}
\caption{The range in $B_2$ of a random walk bridge started at $x$ and conditioned to hit $S_1$ at $y$ is denoted by $\mathcal R_{x,y}$.}
\label{fig:range-Rxy}
\end{figure}

We first complete the proof of the theorem using the lemma. 
By taking $\delta' = \frac{\delta}{k}$ in Lemma~\ref{l:bridgerange}, it is immediate that 
\[
\mathbb E_{\alpha; \{(\widetilde x_i,\widetilde y_i)\}_{i=1}^k}\left[f_2\right]
\leq 
\mathbb E^{\alpha+\delta}\left[f_2\right] + Cks^{2(2-d)}\,\exp\left[-c\textstyle{\frac{\delta}{k}}\,s^{2(d-2)}\right].
\]
We choose $k=\sqrt{\delta}s^{d-2}$, so that $Cks^{2(2-d)}\,\exp\left[-c\frac{\delta}{k}\,s^{2(d-2)}\right]\leq C\exp[\alpha-c\sqrt{\delta}s^{d-2}]$, 
and it remains to show that with this choice of $k$, also 
$\P^\alpha[G_1^c] \leq C\exp[\alpha-c\sqrt{\delta}s^{d-2}]$. 
This follows from Lemma~\ref{l:excursions-AB-num}. Indeed, by \eqref{eq:hitting1}, $\sup_{x\in S_1'}\prw_x\left[H_{S_1}<\infty\right]\leq Cr^{2-d} < \frac1{2e}$ for 
$r$ sufficiently large. Thus, by Lemma~\ref{l:excursions-AB-num}, $\mathbb P^\alpha[G_1^c] \leq \exp[\alpha - k] = \exp[\alpha-\sqrt{\delta}s^{d-2}]$.

This completes the proof of Theorem~\ref{thm:decoupling} subject to Lemma~\ref{l:bridgerange}.
\qed

\subsection{Proof of Lemma~\ref{l:bridgerange}}

Fix $x\in S_1'$ and $y\in S_1$. By \eqref{eq:hitting1}, \eqref{eq:hitting2} and \eqref{eq:hitting3}, the probability that a random walk bridge started at $x$ and conditioned to 
hit $S_1$ in $y$ visits $B_2$ is bounded by 
\[
c\,\left(\frac{r}{s^2}\right)^{d-2}\leq \prw_{x,y}^{S_1}[H_{B_2}<\infty] \leq C\,\left(\frac{r}{s^2}\right)^{d-2},
\]
which is small if $s\geq s_0(d)$ (sufficiently large). In particular, 
\[
\prw_{x,y}^{S_1}[H_{B_2}<\infty] \leq 1 - \exp\left[-2\prw_{x,y}^{S_1}[H_{B_2}<\infty]\right].
\]
Thus, if we denote by $\widetilde {\mathcal R}_{x,y}$ the range in $B_2$ of the Poisson point process $\eta$ of bridges with intensity $\lambda = 2\prw_{x,y}^{S_1}$, then 
$\mathcal R_{x,y}$ is stochastically dominated by $\widetilde {\mathcal R}_{x,y}$, and it suffices to compare $\widetilde{\mathcal R}_{x,y}$ to $\mathcal R$.

Every bridge visits $B_2$ by means of excursions that start on $S_2$ and end on $S_2'$. Let $\eta_m$ be the restriction of $\eta$ to the bridges that 
make exactly $m$ excursions from $S_2$ to $S_2'$. By properties of Poisson point processes, $\eta_m$ are independent Poisson point processes and 
\[
\eta = \sum_{m=0}^\infty \eta_m.
\]
Furthermore, each $\eta_m$ induces a Poisson point process $\sigma_m$ on $m$-tuples of excursions from $S_2$ to $S_2'$, see Figure~\ref{fig:excursions}. 
To describe its intensity measure, let $\mathcal S$ be the set of all finite nearest neighbor paths starting on $S_2$ and 
ending on their first entrance to $S_2'$. For $w_1,\ldots, w_m\in \mathcal S$, $w_i = (w_i(0),\ldots, w_i(k_i))$, let  
\[
\Gamma_m(w_1,\ldots, w_m) = \prod_{i=1}^m \prw_{w_i(0)}\left[(X_0,\ldots, X_{k_i}) = w_i\right]
\prod_{i=1}^{m-1} \prw_{w_i(k_i)}\left[H_{S_2}<H_{S_1},X_{H_{S_2}} = w_{i+1}(0)\right]
\]
be the probability that the excursions from $S_2$ to $S_2'$ made by a simple random walk started at $w_1(0)$ before it ever visits $S_1$ 
are precisely $w_1,\ldots, w_m$. Note that $\Gamma_m$ is a measure on $\mathcal S^m$.
Then, the intensity measure of $\sigma_m$ is 
\begin{multline*}
\lambda_m(w_1,\ldots, w_m) = 2\,\frac{1}{\prw_x\left[X_{H_{S_1}} = y\right]}\,
\prw_x\left[H_{S_2}<H_{S_1}, X_{H_{S_2}} = w_1(0)\right]\\
\Gamma_m(w_1,\ldots, w_m)\,\prw_{w_m(k_m)}\left[H_{S_1}<H_{S_2}, X_{H_{S_1}} = y\right].
\end{multline*}
By \eqref{eq:hitting1}, \eqref{eq:hitting2} and \eqref{eq:hitting3}, if $s$ and $r$ are sufficiently large, then 
\begin{multline}\label{eq:intensity1}
c\,\left(\frac{r}{s^2}\right)^{d-2}\,\widetilde e_{S_2}(w_1(0))\,\Gamma_m(w_1,\ldots, w_m)\leq \lambda_m(w_1,\ldots, w_m)\\ 
\leq C\,\left(\frac{r}{s^2}\right)^{d-2}\,\widetilde e_{S_2}(w_1(0))\,\Gamma_m(w_1,\ldots, w_m).
\end{multline}
We would like to compare $\lambda_m$ with the intensity measure of 
the Poisson point process of $m$-tuples of excursions from $S_2$ to $S_2'$ induced by the Poisson point process $\mathscr L^{\delta'}_{S_2,S_2'}$ of 
loops that visit $S_2$ and $S_2'$. A slight problem is that these loop excursions are only defined up to a cyclic permutation. 
To avoid this issue, we use Lemma~\ref{l:LalphaAB-basedloops}, which states that 
the Poisson point process $\mathscr L^{\delta'}_{S_2,S_2'}$ can be constructed by (a) sampling the Poisson point process $\eta'$ of \emph{based} loops with intensity measure 
\[
\dot{\ell} \mapsto \delta'\,\frac{1}{k(\dot{\ell})}\,\prw_{x_0}\left[X_i = x_i, 0\leq i\leq |\dot\ell|\right]\,\mathds{1}_{\dot{\mathfrak L}_{S_2,S_2'}}(\dot\ell),\qquad \dot\ell = (x_0,\ldots, x_{|\dot\ell|}),
\]
and (b) ``forgetting'' the location of the root. 
In particular, the ranges in $B_2$ of loops from $\mathscr L^{\delta'}$ that visit both $S_2$ and $S_2'$ and 
that of loops from $\eta'$ have the same distribution. The excursions of loops in $\eta'$ are naturally ordered. 
\begin{figure}[!tp]
\centering
\resizebox{16cm}{!}{\includegraphics{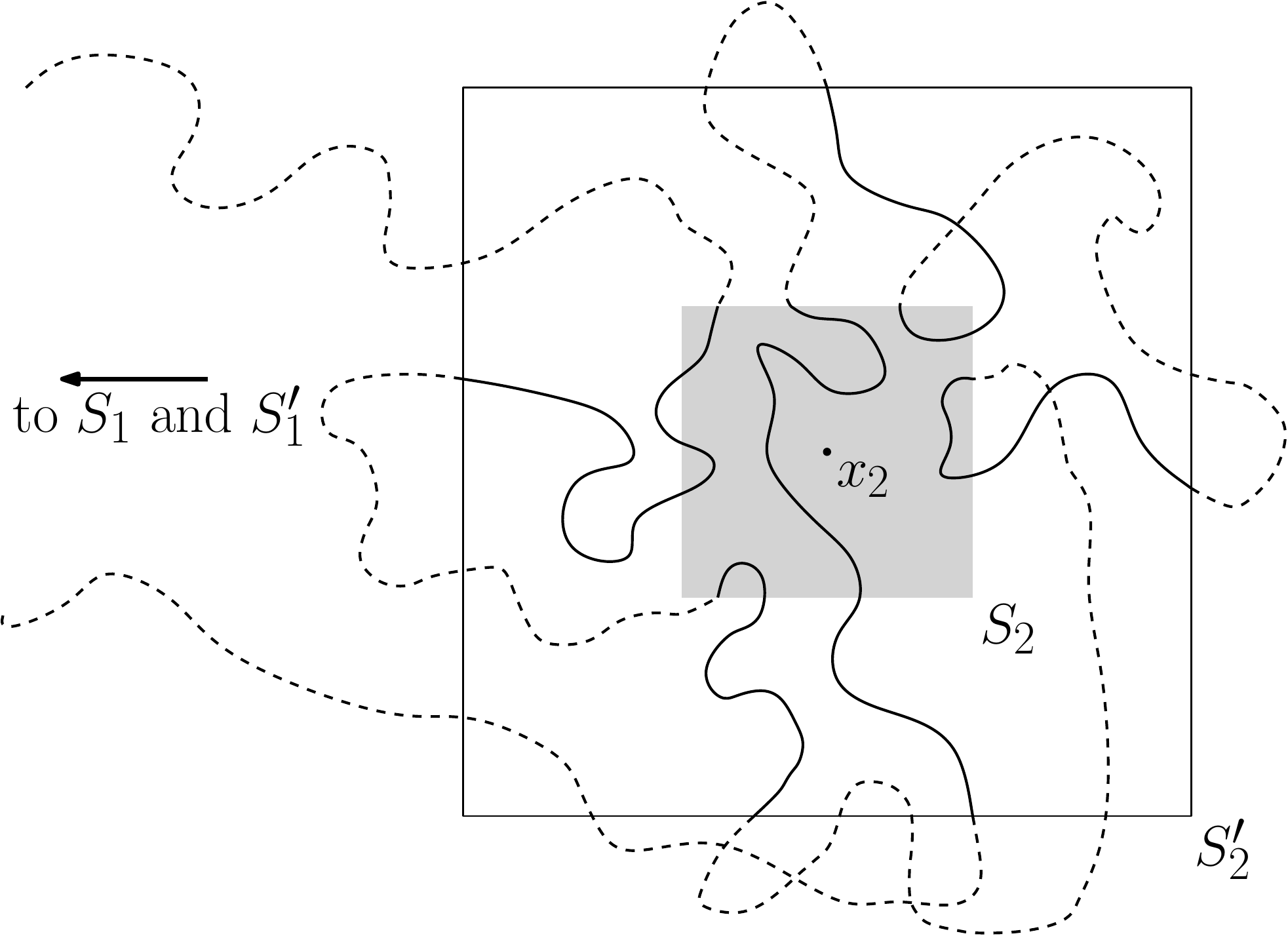}\hspace{8em}\includegraphics{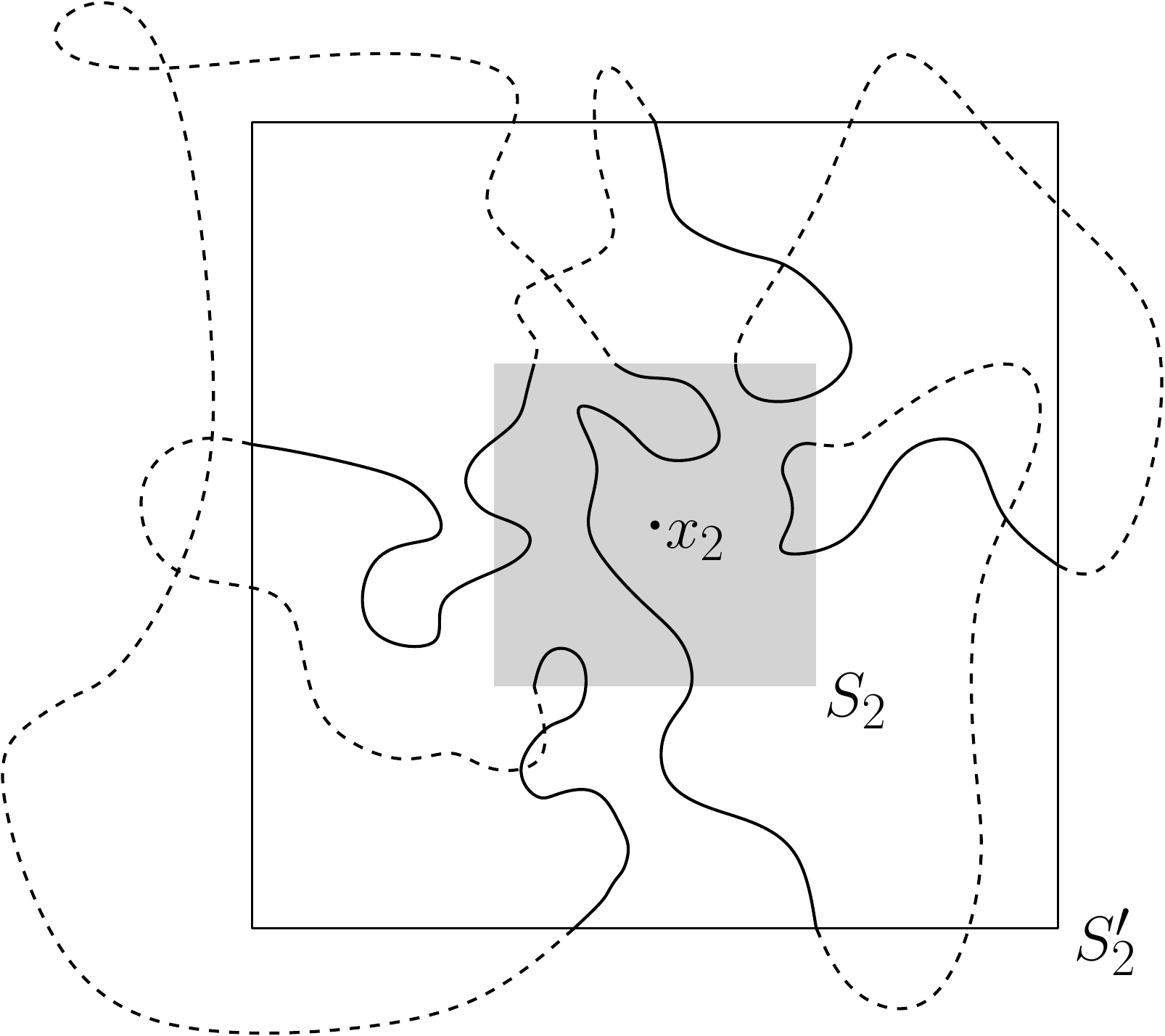}}
\caption{On the left, a $5$-tuple of excursions from $S_2$ to $S_2'$ induced by a random walk bridge from $\eta_5$, on the right, 
by a random walk loop from $\eta_5'$.} 
\label{fig:excursions}
\end{figure}
Let $\eta_m'$ be the restriction of $\eta'$ to the loops that make exactly $m$ excursions, 
then $\eta_m'$ are independent Poisson point processes and $\eta' = \sum_{m=1}^\infty \eta_m'$. 
Furthermore, $\eta_m'$ induces a Poisson point process $\sigma_m'$ on $m$-tuples of excursions (see Figure~\ref{fig:excursions}) with intensity measure 
\[
\lambda_m'(w_1,\ldots, w_m) = \delta'\,\frac{1}{m}\,\Gamma_m'(w_1,\ldots, w_m)\,\prw_{w_m(k_m)}\left[X_{H_{S_2}}=w_1(0)\right],
\]
where 
\[
\Gamma_m'(w_1,\ldots, w_m) = \prod_{i=1}^m \prw_{w_i(0)}\left[(X_0,\ldots, X_{k_i}) = w_i\right]\,
\prod_{i=1}^{m-1} \prw_{w_i(k_i)}\left[X_{H_{S_2}} = w_{i+1}(0)\right].
\]
In particular, by Lemma~\ref{l:hitting}, 
\begin{multline}\label{eq:intensity2}
c\,r^{2-d}\, \delta'\,\frac{1}{m}\, \widetilde e_{S_2}(w_1(0)) \Gamma_m'(w_1,\ldots, w_m)\leq \lambda_m'(w_1,\ldots, w_m)\\ 
\leq C\,r^{2-d}\, \delta'\,\frac{1}{m}\, \widetilde e_{S_2}(w_1(0)) \Gamma_m'(w_1,\ldots, w_m).
\end{multline}
It is immediate that $\Gamma_m\leq \Gamma_m'$. Thus, by \eqref{eq:intensity1} and \eqref{eq:intensity2}, $\lambda_m\leq \lambda_m'$ if 
\begin{equation}\label{eq:m}
m\leq c\,\left(\frac{s}{r}\right)^{2(d-2)}\,\delta',
\end{equation}
which implies that for these $m$'s, $\sigma_m$ is stochastically dominated by $\sigma_m'$. 
In particular, if $\sigma_m = 0$ for all $m> c\,\left(\frac{s}{r}\right)^{2(d-2)}\,\delta'$, then $\sum_{m=1}^\infty\sigma_m$ is stochastically dominated by $\sum_{m=1}^\infty\sigma_m'$. 

\smallskip

Let $G_2$ be the event that $\sigma_m = 0$ for all $m> c\,\left(\frac{s}{r}\right)^{2(d-2)}\,\delta'$. It follows that 
there exists a coupling $(\overline{\mathcal R}_{x,y},\overline {\mathcal R})$ of $\mathcal R_{x,y}$ and $\mathcal R$ such that 
\[
\P\left[\overline{\mathcal R}_{x,y}\subseteq \overline {\mathcal R}\right]\geq \P[G_2].
\]
Finally, for each $m$, using \eqref{eq:intensity1} and Lemma~\ref{l:hitting}, 
\[
\P[\sigma_m \neq 0]\leq \lambda_m[\mathcal S^m] \leq C\,\left(\frac{r}{s^2}\right)^{d-2}\,\left(C\,r^{2-d}\right)^{m-1}.
\]
Thus, by choosing $r$ sufficiently large (depending only on the dimension), 
\[
\P[G_2^c]\leq Cs^{2(2-d)}\,\exp\left[-c\delta'\,s^{2(d-2)}\right],
\]
which completes the proof of the lemma.
\qed

\begin{remark}\label{rem:decoupling:proof}[Some comments on the proof of Theorem~\ref{thm:decoupling}] 
The following observations suggest that the error term of \eqref{eq:decoupling} and \eqref{eq:decoupling:decreasing} 
could not be improved with our method. 
\begin{enumerate}[leftmargin=*]
\item 
The estimate \eqref{eq:G1-G1c} may at first look rather crude. It seems better to consider events $F_k = \{\mathcal Z = k\}$ and write
\[
\mathbb E^\alpha\left[f_1\, f_2\right] =
\sum_{k=0}^\infty\,\mathbb E^\alpha\left[f_1\mathds{1}_{F_k}\,\mathbb E^\alpha\left[f_2~|~\mathcal E\right]\right].
\]
However, using Lemma~\ref{l:bridgerange} to bound $\mathbb E^\alpha\left[f_2~|~\mathcal E\right]$ and the exact asymptotics of $\mathbb P^\alpha[F_k]$, 
one would get the error term in the form $\sum_{k=1}^\infty \exp\left(-ck - c\frac{\delta}{k}\,s^{2(d-2)}\right)$, 
which is precisely of the order $\exp(-c\sqrt{\delta}\,s^{d-2})$.
\item
In the comparison of intensity measures $\lambda_m$ and $\lambda_m'$ in the proof of Lemma~\ref{l:bridgerange} we use the trivial bound $\Gamma_m\leq \Gamma_m'$, 
which allows to conclude that $\lambda_m\leq\lambda_m'$ only for $m$ satisfying \eqref{eq:m}. 
By taking into account the information that the random walk bridge does not return to $S_1$ between the excursions $w_i$, one can show that for every $m$, 
\[
\Gamma_m \leq \left(1 - c\left(\frac{r}{s^2}\right)^{d-2}\right)^{m-1}\,\Gamma_m'. 
\]
This gives no improvement to the trivial bound for the $m$s of the order $\left(\frac{s}{r}\right)^{2(d-2)}\,\delta'$, although it does imply $\lambda_m\leq\lambda_m'$ for all large enough $m$. 

Incidentally, using this better comparison of $\Gamma_m$ and $\Gamma_m'$ one obtains that $\lambda_m \leq c\,r^{2-d}\,\delta'\,\lambda_m'$ for every $m$. 
In particular, if $\delta'\geq C\,r^{d-2}$, then the range of the random walk bridge in $B_2$ is stochastically dominated by the range of the loop soup $\mathscr L^{\delta'}$ (with probability $1$). 
\end{enumerate}
\end{remark}

\begin{remark}
The arguments of the proof of Theorem~\ref{thm:decoupling} apply also to loop soups of random walks with general bounded jump distributions considered in \cite{LL10} 
as well as to the Brownian loop soup defined in \cite{LW04}, leading to analogous decoupling inequalities for these models. 
\end{remark}

\section{Proof of Theorem~\ref{thm:lu}}\label{sec:proof-local-uniqueness}

The overall idea of the proof is similar to that of \cite{DRS-AIHP}, where a result analogous to Theorem~\ref{thm:lu} is proven for the 
vacant set of random interlacements, although the implementations are quite different. 
As in \cite{DRS-AIHP} we partition the lattice $\Z^d$ into good and bad boxes. 
Each good box has a vacant ``frame'' (see Definition~\ref{def:edges}) and uniformly bounded cumulative occupation local times for $\mathscr L^\alpha$. 
In Proposition~\ref{prop:bad-*paths} and Corollary~\ref{cor:bad:*paths} we prove that 
the set of good boxes typically contains an infinite connected component, 
whose complement consists only of small holes. When it is the case, any vacant path of big diameter will pass through a large number of good boxes. 
However, each time the path enters a good box, there is a uniformly positive probability that it locally connects to the frame of the good box, 
as proved in Lemma~\ref{l:surgery}, which makes the existence of long isolated vacant paths unlikely. 

Let us indicate the key differences of our approach from that in \cite{DRS-AIHP}. 
The existence of a ubiquitous infinite cluster of good boxes is proven in \cite{DRS-AIHP} using in an essential way 
a strong version of decoupling inequalities for random interlacements (see \cite[Theorem~7.2]{DRS-AIHP}). 
Because of an explicit and very specific dependence of the error term on the intensity of random interlacements and relevant scales (see \cite[(7.5)]{DRS-AIHP}), 
these decoupling inequalities imply a qualitative bound on the probabilities of cascading events 
under the assumption that a box of size $L_0$ is unlikely to be bad for the random interlacements with intensity $L_0^{2-d}$ for large $L_0$ (see \cite[Lemma~2.2]{DRS-AIHP}), 
which is verified in \cite[Lemma~3.5]{DRS-AIHP}. The ubiquity of good boxes then follows easily from \cite[Lemma~2.2]{DRS-AIHP}, see \cite[Lemma~3.6]{DRS-AIHP}. 

There are several issues in adapting this approach to our setting. 
The decoupling inequalities \eqref{eq:decoupling} and \eqref{eq:decoupling:decreasing} are weaker than the ones in \cite[Theorem~7.2]{DRS-AIHP} 
(e.g., the latter imply the decoupling inequalities \ppp{}, which are not available for the loop soup, cf.\ Remark~\ref{rem:p3-loopsoup}). 
Still, they do give an analogue of \cite[Lemma~2.2]{DRS-AIHP} under a stronger assumption that large boxes are unlikely to be bad for the loop soup with a fixed intensity 
(see Theorem~\ref{thm:badseed:proba}). 
This assumption cannot be true for the loop soup though, predominantly because of the positive density of small loops. 

Instead of trying to solve these issues (which, even if successful, would only give \eqref{eq:lu:1} and \eqref{eq:lu:2} with probability $\geq 1 - C\exp(-(\log n)^{1+\epsilon})$, 
since the scales $L_n$ in Theorem~\ref{thm:badseed:proba} grow faster than exponentially), 
we develop an approach that does not rely on decoupling inequalities.
We use an idea from \cite{Teixeira} adapted to our setting to bound the probability that a suitably spread out family of 
boxes consists only of bad ones (see Lemma~\ref{l:bad-leaves}) directly using the decomposition of loops into excursions (Proposition~\ref{prop:sampling-loopsoup}) and 
the large deviation bound on the number of excursions (Lemma~\ref{l:excursions-AB-num}). 
This approach may be of independent interest, since it could potentially apply to models, 
for which decoupling inequalities are not available or have not been developed yet, such as, e.g., the voter percolation \cite{RV17}.

\smallskip

Fix an integer $R\geq 1$, let $L_0 = 2R+1$ and consider the lattice 
\[
\GG_0 = L_0\,\Z^d
\]
with edges between any $\ell_1$ nearest neighbor vertices of $\GG_0$. 
If $x',y'\in\GG_0$ are neighbors, we write $x'\nngg y'$. 
For $n\in\N$ and $x'\in\GG_0$, let $\bb(x',n) = \{y'\in\GG_0: \|x'-y'\|\leq L_0n\}$ and 
$\ss(x',n) = \{y'\in\GG_0: \|x'-y'\|= L_0n\}$ be the $\ell_\infty$ ball, resp., sphere, of radius $n$ in $\GG_0$ centered at $x'$.

For $x'\in\GG_0$, define 
\[
\cube(x') = \ball(x',R).
\]
Then, $\{\cube(x'),\, x'\in \GG_0\}$ is a partition of $\Z^d$ into disjoint hypercubes.

\begin{definition}\label{def:edges}
Let $\edges$ be the subset of $\cube(0)$ that consists of all vertices having at least two of their coordinates 
in the set $\{-R,-R+1,-R+2,R-2,R-1,R\}$ and define 
\[
\edges(x') = x' + \edges,\quad x'\in\GG_0.
\]
(For $d=3$, $\edges(x')$ is just the $\ell_\infty$ $2$-neighborhood of the edges of the cube $\cube(x')$.) 
\end{definition}
Note that 
\begin{itemize}\itemsep0pt
\item
the set $\edges$ is connected in $\Z^d$, 
\item
for any $x_1'\nngg x_2'\in\GG_0$, the set $\edges(x_1')\cup\edges(x_2')$ is connected in $\Z^d$, 
\end{itemize}

Any function $\lt:\Z^d\to\N_0 = \{0,1,\ldots\}$ gives a decomposition of $\GG_0$ into good and bad vertices:
\begin{definition}\label{def:good}
Let $\lt:\Z^d\to\N_0$. Vertex $x'\in\GG_0$ is \emph{$R$-good} for $\lt$ if 
\begin{enumerate}\itemsep1pt
\item[(1)]
$\lt(x) = 0$ for all $x\in\edges(x')$,
\item[(2)]
$\sum_{x\in\dint\cube(x')} \lt(x) \leq R^{d-1}$. 
\end{enumerate}
Otherwise, $x'$ is \emph{$R$-bad} for $\lt$. 
\end{definition}
\begin{remark}\label{rem:good}
In our applications, $\sum_{x\in\dint\cube(x')} \lt(x)$ will correspond to the number of times a finite collection of independent random walks visit 
$\dint\cube(x')$, cf.\ \eqref{eq:Rd-1}. Thus, $R^{d-1}$ in Definition~\ref{def:good}(2) could be replaced by any $f(R)\gg R$. 
\end{remark}
We write 
\[
\begin{split}
\good(\lt) &= \{x'\in\GG_0~:~x'\text{ is $R$-good for $\lt$}\},\\
\bad(\lt) &= \{x'\in\GG_0~:~x'\text{ is $R$-bad for $\lt$}\}.
\end{split}
\]

\medskip

The choice of $\alpha_1>0$ in Theorem~\ref{thm:lu} is made in the following proposition, 
which is proven in Section~\ref{sec:prop:bad-*paths}. 
Recall that $\ltloop^\alpha$ denotes the field of local times of the loop soup $\mathscr L^\alpha$.
\begin{proposition}\label{prop:bad-*paths}
For any $d\geq 3$, there exist $R\geq 1$, $\alpha_1>0$, $c>0$ and $C<\infty$ such that 
for all $\alpha\leq \alpha_1$ and $N\geq 1$, 
\begin{equation}\label{eq:bad-*paths}
\P\left[
\text{$0$ is $*$-connected to $\ss(0,N)$ in $\bad(\ltloop^\alpha)$}
\right]
\leq C\,\exp\left(-N^c\right) .
\end{equation}
(Sets $X,Y\subset\GG_0$ are $*$-connected in $Z\subset\GG_0$ if there exist $z_0,\ldots,z_n\in Z$ such that 
$z_0\in X$, $z_n\in Y$ and $\|z_i-z_{i-1}\|=L_0$ for all $1\leq i\leq n$.)
\end{proposition}

Proposition~\ref{prop:bad-*paths} easily implies the existence of 
(a) unique infinite component $\mathcal G^\alpha_\infty$ in $\good(\ltloop^\alpha)$ and 
(b) ubiquitous connected component $\mathcal G^\alpha_N$ in $\good(\ltloop^\alpha)\cap \bb(0,N)$:

\begin{corollary}\label{cor:bad:*paths}
Fix $R\geq 1$ and $\alpha_1>0$ as in Proposition~\ref{prop:bad-*paths}. 
There exist $c'=c'(d)>0$, and $C'=C'(d)<\infty$ such that for all $\alpha\leq \alpha_1$ and $N\geq 1$, 
\begin{enumerate}
\item[(a)]
there exists a unique infinite connected (in $\GG_0$) component $\mathcal G^\alpha_\infty$ of $\good(\ltloop^\alpha)$ and  
\begin{equation}\label{eq:Ginfty:ball}
\mathbb P\left[\mathcal G^\alpha_\infty\cap \bb(0,N)\neq\emptyset \right]
\geq 1 - C'\, \exp\left(-N^{c'}\right) , 
\end{equation}
\item[(b)]
if $\mathcal G^\alpha_N$ denotes a unique connected component of $\good(\ltloop^\alpha)\cap \bb(0,N)$ such that 
any nearest neighbor path in $\GG_0$ from any $x'\in \bb(0,\lfloor\frac23N\rfloor)$ to $\ss(x',\lfloor\frac{1}{30}N\rfloor)$ 
intersects $\mathcal G^\alpha_N$ at least $\sqrt{N}$ times, or the empty set if such component does not exist, then 
\begin{equation}\label{eq:SN:ball}
\mathbb P\left[\mathcal G^\alpha_N\neq \emptyset\right] \geq 1 - C'\,\exp\left(- N^{c'}\right).
\end{equation}
\end{enumerate}
\end{corollary}
\begin{proof}[Proof of Corollary~\ref{cor:bad:*paths}]
The proof is essentially the same as the proof of \cite[Corollary~3.7]{DRS-AIHP}, where the role of Proposition~\ref{prop:bad-*paths} 
is played by \cite[Lemma~3.6]{DRS-AIHP}. We omit the details. 
\end{proof}

\begin{proof}[Proof of Theorem~\ref{thm:lu}]
The first statement \eqref{eq:lu:1} follows immediately from \eqref{eq:Ginfty:ball}, since the set 
$\bigcup_{x'\in \mathcal G^\alpha_\infty}\edges(x')$ is an infinite connected subset of the vacant set $\mathcal V^\alpha$. 
To prove \eqref{eq:lu:2}, by the union bound, it suffices to show that for each $x\in \ball(0,L_0\lfloor\frac23N\rfloor)$, 
\begin{equation}\label{eq:x-to-SalphaN}
\mathbb P\left[\begin{array}{c}
\text{$x$ is connected to $\Z^d\setminus \ball(x,L_0\lfloor\frac{1}{25}N\rfloor)$ in $\mathcal V^\alpha$,}\\[4pt] 
\text{but not to $\bigcup_{x'\in\mathcal G^\alpha_N}\edges(x')$ in $\mathcal V^\alpha\cap \ball(0,L_0N+R)$}
\end{array}\right]
\leq C''\,\exp\left(- N^{c''}\right).
\end{equation}
(To link \eqref{eq:x-to-SalphaN} to \eqref{eq:lu:2}, one can take $N = \lfloor \frac{2n-R}{L_0}\rfloor$, 
then $L_0\lfloor\frac23N\rfloor\geq n$ and $L_0\lfloor\frac{1}{25}N\rfloor\leq \frac{1}{10}n$ for all large enough $n$.)

\smallskip

The main idea of the proof of \eqref{eq:x-to-SalphaN} is to explore the connected component of $x$ in $\mathcal V^\alpha$ progressively in boxes $\cube(x')$, $x'\in\GG_0$.  
If the ubiquitous component $\mathcal G^\alpha_N$ of good vetices is not empty, then the cluster of $x$ will encounter at least $\sqrt N$ boxes 
centered at vertices from $\mathcal G^\alpha_N$. Each time the encounter happens, excluding possibly the very first box, the explored part of the cluster of $x$ 
connects locally to the set $\bigcup_{x'\in\mathcal G^\alpha_N}\edges(x')$ with probability at least $\gamma$ uniformly over all possible realizations of good boxes 
and the explored history. This will lead to the upper bound $(1-\gamma)^{\sqrt{N}-1}$.

\smallskip

The lower bound on the conditional probability of the local connectedness to $\bigcup_{x'\in\mathcal G^\alpha_N}\edges(x')$ after each step of exploration follows from Lemma~\ref{l:surgery} below.
For $R\geq 1$ and $\alpha>0$, denote by 
\[
\Sigma_{\mathcal G} = \sigma\left(\mathds{1}_{x'\in \good(\ltloop^\alpha)},\,x'\in\GG_0\right)
\]
the $\sigma$-algebra generated by all the good boxes for $\ltloop^\alpha$. (Note that $\mathcal G^\alpha_N$ is $\Sigma_{\mathcal G}$-measurable.)
For any $x'\in\GG_0$, denote by 
\[
\mathcal A_{x'} = \sigma\left(\mathds{1}_{z\in\mathcal V^\alpha},\,z\notin \cube(x')\right)
\]
the $\sigma$-algebra generated by the vacant set $\mathcal V^\alpha$ (equivalently, by the range) of the loop soup $\mathscr L^\alpha$ outside of the box $\cube(x')$. 
\begin{lemma}\label{l:surgery}
Let $d\geq 3$, $R\geq 1$ and $\overline\alpha>0$. 
There exists $\gamma = \gamma(d,R,\overline\alpha)>0$ such that for all $\alpha\in(0,\overline\alpha]$, $x'\in\GG_0$ and $y\in\dext\cube(x')$, 
\begin{equation}\label{eq:surgery}
\mathbb P\left[
\text{$y$ is connected to $\edges(x')$ in $\mathcal V^\alpha\cap(\{y\}\cup\cube(x'))$}
\,\Big|\,\Sigma_{\mathcal G},\,\mathcal A_{x'} \right] \geq \gamma\,\mathds{1}_{y\in\mathcal V^\alpha, x'\in\good(\ltloop^\alpha)},\,\,\,\, \mathbb P\text{-a.s.}
\end{equation}
\end{lemma}
We postpone the proof of Lemma~\ref{l:surgery} to Section~\ref{sec:surgery} and now complete the proof of Theorem~\ref{thm:lu} using the lemma.

\smallskip

Fix $x\in \ball(0,L_0\lfloor\frac23N\rfloor)$. 
We now define the algorithm for the exploration of the connected component of $x$ in $\mathcal V^\alpha$ 
which progressively reveals $\mathcal V^\alpha$ in boxes $\cube(x')$, $x'\in\GG_0$. 
Assume that the vertices of $\Z^d$ are ordered lexicographically. 
\begin{itemize}
\item
Let $x_0'\in\GG_0$ be the unique vertex such that $x\in\cube(x_0')$ and define $A_0 = \cube(x_0')$. 
(Necessarily, $x_0'\in\bb(0,\lfloor\frac23N\rfloor)$.)
\item
Let $k\geq 0$ and assume that $x_k'$ and $A_k$ are determined. 
We stop the algorithm if 
\[
\text{(a) $x_k'\in \ss(x_0',\lfloor{\textstyle\frac{1}{30}}N\rfloor)$ or (b) $x$ is not connected to $\dint A_k$ in $\mathcal V^\alpha$,}
\]
and define $\tau = k$, $y_l = y_k$, $x_l' = x_k'$, $A_l = A_k$, for all $l>k$. 

Else, we define
\begin{itemize}
\item
$y_{k+1}\in \dint A_k$ as the smallest vertex such that $x$ is connected to $y_{k+1}$ in $\mathcal V^\alpha\cap A_k$,
\item
$x_{k+1}'\in\GG_0\setminus\{x_0',\ldots,x_k'\}$ as the smallest vertex such that $y_{k+1}\in \dext \cube(x_{k+1}')$, 
\item
$A_{k+1} = A_k\cup \cube(x_{k+1}')$.
\end{itemize}
(See Figure~\ref{fig:algorithm} for an illustration.)
\end{itemize}
\begin{figure}[!tp]
\centering
\resizebox{7cm}{!}{\includegraphics{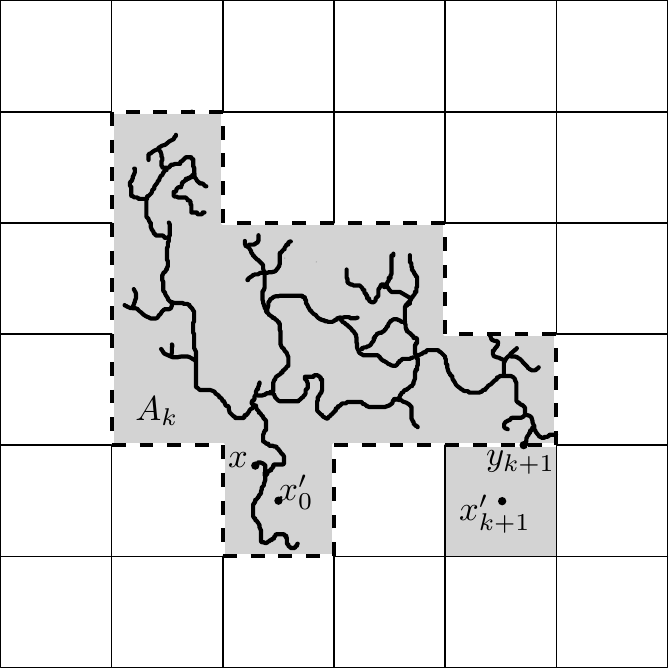}}
\caption{Exploration algorithm.}
\label{fig:algorithm}
\end{figure}
The algorithm always stops in a finite time (which we denote by $\tau$), and if $x$ is connected to 
$\Z^d\setminus \ball(x,L_0\lfloor\frac{1}{25}N\rfloor)$ in $\mathcal V^\alpha$, 
then the algorithm stops exactly on ``reaching'' $\ss(x_0',\lfloor\frac{1}{30}N\rfloor)$.

\smallskip

Consider the sigma-algebras 
\[
\mathcal A_k = \sigma\left(A_k,\,\mathcal V^\alpha\cap A_k\right)\quad \text{and}\quad 
\mathcal Z_k = \sigma\left(\sigma(\mathcal G^\alpha_N), \mathcal A_k\right),\qquad k\geq 0.
\]
Note that the random elements $y_i$, $x_i'$, $A_i$, for $1\leq i\leq k$, are $\mathcal A_{k-1}$-measurable, 
since by revealing the shape of $A_{k-1}$ and the state of $\mathcal V^\alpha$ in $A_{k-1}$, 
one can reconstruct the steps $1,\ldots, k-1$ of the algorithm uniquely and 
also uniquely determine $y_k$, $x_k'$ and $A_k$. 
Same reasoning gives that the event $\{\tau\geq k\}$ belongs to $\mathcal A_{k-1}$.

\smallskip

Consider the events 
\[
E_k = \left\{\tau\geq k,\,x_k'\in\mathcal G^\alpha_N,\,y_k\text{ is connected to }\edges(x_k')\text{ in }\mathcal V^\alpha\cap\left(\{y_k\}\cup\cube(x_k')\right)\right\},\quad k\geq 1.
\]
Then $E_k\in \mathcal Z_k$, $\{\tau\geq k,\,x_k'\in\mathcal G^\alpha_N\}\in\mathcal Z_{k-1}$, and 
\begin{equation}\label{eq:Ek}
\mathbb P\left[E_k\,\Big|\,\mathcal Z_{k-1}\right]\geq
\gamma\,\mathds{1}_{\tau\geq k,\,x_k'\in\mathcal G^\alpha_N},\quad \mathbb P\text{-a.s.,}
\end{equation}
with $\gamma$ as in Lemma~\ref{l:surgery} (for $\overline\alpha=\alpha_1$). Indeed, to see that \eqref{eq:Ek} holds, fix $k\geq 1$ and for any admissible $G$, $A$ and $V$, 
define the event $F(G,A,V) = \{\mathcal G^\alpha_N = G,\,A_{k-1}=A,\,\mathcal V^\alpha\cap A_{k-1} = V\}$. 
Note that if $F(G,A,V)$ occurs, then $x_k'=x'$ and $y_k=y$ for some $x'$ and $y$, which are uniquely determined by $A$ and $V$. Thus, 
\begin{eqnarray*}
\mathbb P\left[E_k,\,F(G,A,V)\right] &= &\mathbb P\left[ E_k,\,F(G,A,V),\,x_k' = x',\,y_k=y\right]\\
&=
&\mathbb E\left[\mathbb P\left[E_k,\,F(G,A,V),\,x_k'=x',\,y_k=y\,\Big|\,\Sigma_\mathcal G,\,\mathcal A_{x'}\right]\right]\\
&=
&\mathbb E\Big[ F(G,A,V),\,x_k'=x',\,y_k=y,\,\tau\geq k,\,x_k'\in\mathcal G^\alpha_N,\\
&&\quad\mathbb P\left[y\text{ is connected to }\edges(x')\text{ in }\mathcal V^\alpha\cap\left(\{y\}\cup\cube(x')\right)\,\Big|\,\Sigma_\mathcal G,\,\mathcal A_{x'}\right]\Big]\\
&\stackrel{\eqref{eq:surgery}}\geq
&\gamma\,\mathbb P\Big[F(G,A,V),\,x_k'=x',\,y_k=y,\,\tau\geq k,\,x_k'\in\mathcal G^\alpha_N\Big]\\
&= &\gamma\,\mathbb P\Big[F(G,A,V),\,\tau\geq k,\,x_k'\in\mathcal G^\alpha_N\Big],\qquad\text{which proves \eqref{eq:Ek}.}
\end{eqnarray*}

\bigskip

We can now complete the proof of \eqref{eq:x-to-SalphaN}. Let 
\begin{align*}
\tau_1 &= \inf\{k\geq 1\,:\,x_k'\in\mathcal G^\alpha_N\}\\
\tau_i &= \inf\{k>\tau_{i-1}\,:\,x_k'\in\mathcal G^\alpha_N\},\,\text{for }i\geq 2.
\end{align*}
Note that $\{\tau_i= k\}\in\mathcal Z_{k-1}$ for all $i$ and $k$. Let $M=\lfloor \sqrt{N}\rfloor -1$. Then, 
the probability on the left hand side of \eqref{eq:x-to-SalphaN} is bounded from above by 
\begin{multline*}
\leq \mathbb P\left[\mathcal G^\alpha_N= \emptyset\right] + \mathbb P\left[\bigcap_{i=1}^M E_{\tau_i}^c,\,\tau_M\leq\tau\right]
= \mathbb P\left[\mathcal G^\alpha_N= \emptyset\right] + \sum_{k=1}^\infty\mathbb P\left[\bigcap_{i=1}^M E_{\tau_i}^c,\,\tau_M = k\leq\tau\right]\\
= \mathbb P\left[\mathcal G^\alpha_N= \emptyset\right] + 
\sum_{k=1}^\infty\mathbb E\left[\bigcap_{i=1}^{M-1} E_{\tau_i}^c,\,\tau_M = k\leq\tau,\,x_k'\in\mathcal G^\alpha_N,\,\mathbb P\left[E_k^c\,\Big|\,\mathcal Z_{k-1}\right]\right]\\
\stackrel{\eqref{eq:Ek}}\leq \mathbb P\left[\mathcal G^\alpha_N= \emptyset\right] + 
(1-\gamma)\,\sum_{k=1}^\infty\mathbb P\left[\bigcap_{i=1}^{M-1} E_{\tau_i}^c,\,\tau_M = k\leq\tau,\,x_k'\in\mathcal G^\alpha_N\right]\\
\leq \mathbb P\left[\mathcal G^\alpha_N= \emptyset\right] + 
(1-\gamma)\,P\left[\bigcap_{i=1}^{M-1} E_{\tau_i}^c,\,\tau_{M-1} \leq\tau\right]\\
\leq \ldots \leq \mathbb P\left[\mathcal G^\alpha_N= \emptyset\right] + (1-\gamma)^M.
\end{multline*}
An application of \eqref{eq:SN:ball} completes the proof of \eqref{eq:x-to-SalphaN} and thus of \eqref{eq:lu:2}, 
subject to Proposition~\ref{prop:bad-*paths} and Lemma~\ref{l:surgery}.
\end{proof}

\subsection{Proof of Proposition~\ref{prop:bad-*paths}}\label{sec:prop:bad-*paths}

The proof uses a multiscale analysis and embedding of dyadic trees.  
Its main idea is similar to the proof of \cite[Theorem~3.2]{Teixeira} about random interlacements, 
although we use embeddings of dyadic trees as in \cite{Sznitman:Decoupling, Rath} instead of skeletons as in \cite{Teixeira}. 
After defining the embeddings and proving some of their relevant properties 
(detailed proofs of various results about such embeddings can be found in \cite{Rath}) 
we prove in Lemma~\ref{l:bad-leaves} that an embedding into the set $\bad(\mathcal N^\alpha)$ of bad vertices is very unlikely. 
Since the connection event in \eqref{eq:bad-*paths} implies that such an embedding must exist (within a not too big class of embeddings), 
it must be very unlikely too. 

\medskip

We proceed with the details. 
Recall that $L_0=2R+1$. 
Let $l\geq 1$ be an integer and consider the sequence of geometrically growing scales $L_n = L_0\,l^n$, $n\geq 0$, 
and respective lattices $\GG_n = L_n\,\Z^d$. 

\smallskip

For $n\geq 0$, we denote by $T_n = \bigcup_{k=0}^n\{1,2\}^k$ the dyadic tree of depth $n$ and write $T_{(k)} = \{1,2\}^k$ for the collection of 
elements of the tree at depth $k$. 
Let $\Lambda_n$ be the set of embeddings $\mathcal T:T_n\to\Z^d$ such that 
\begin{itemize}\itemsep0pt
\item
$\mathcal T(\emptyset) = 0$,
\item
for all $1\leq k\leq n$ and $m\in T_{(k)}$, $\mathcal T(m)\in\GG_{n-k}$,
\item
for all $0\leq k\leq n-1$, $m\in T_{(k)}$ and $i\in\{1,2\}$, 
\begin{equation}\label{eq:Tmi-Tm}
\|\mathcal T(mi) - \mathcal T(m)\| = i\,L_{n-k}.
\end{equation}
(Here $mi\in T_{(k+1)}$ is a descendant of $m$.)
\end{itemize}
\begin{lemma}\label{l:Lambda-n}
For all $n\geq 1$, $L_0\geq 1$, $l\geq 1$, 
\begin{enumerate}[leftmargin=*]
\item
$|\Lambda_n| \leq \left((2d\,(2l+1)^{d-1})\,(2d\,(4l+1)^{d-1})\right)^{2^n-1} \leq \left((2d)^2\,(4l)^{2(d-1)}\right)^{2^n-1}$,
\item
for all $\mathcal T\in \Lambda_n$, $k\geq 0$ and $m\in T_{(n)}$, 
\[
\left|\left\{m'\in T_{(n)}~:~\|\mathcal T(m') - \mathcal T(m)\|\leq \frac{l-5}{l-1}\, L_{k+1}\right\}\right|\leq 2^k.
\]
\end{enumerate}
\end{lemma}
\begin{proof}[Proof of Lemma~\ref{l:Lambda-n}]
Statement 1 follows easily by induction on $n$. 

For Statement 2, it suffices to consider $0\leq k\leq n-1$ and $l\geq 6$. Take $a\in T_{(n-k-1)}$ and $b',b''\in\{1,2\}^k$. Then for the elements $a1b',a2b''\in T_{(n)}$, 
\begin{multline*}
\left\|\mathcal T(a1b') - \mathcal T(a2b'')\right\|\\[4pt]
\geq \left\|\mathcal T(a1) - \mathcal T(a2)\right\| - \left\|\mathcal T(a1b') - \mathcal T(a1)\right\| - \left\|\mathcal T(a2b'') - \mathcal T(a2)\right\|\\[4pt]
\stackrel{\eqref{eq:Tmi-Tm}}\geq L_{k+1} - 2\,\left(2L_{k} + 2L_{k-1} + \ldots + 2L_0\right) 
> L_{k+1} - 4L_k\,\frac{l}{l-1} = \frac{l-5}{l-1}\,L_{k+1}.
\end{multline*}
Thus, any $m,m'\in T_{(n)}$ with $\|\mathcal T(m') - \mathcal T(m)\|\leq\frac{l-5}{l-1}\,L_{k+1}$ 
can only differ in the last $k$ digits, i.e., there exist $a\in T_{(n-k)}$, $b,b'\in\{1,2\}^k$ such that $m=ab$ and $m'=ab'$. 
Since for any $m$ there are at most $2^k$ such $m'$, the result follows. 
\end{proof}

\medskip

For $x'\in\GG_0$, define
\[
C_{x'} = \dint \ball(x',L_0),\quad D_{x'} = \dint \ball(x',{\textstyle\frac{1}{4}}L_1)
\]
and for $\mathcal T\in \Lambda_n$, consider 
\[
C_{\mathcal T} = \bigcup_{x'\in \mathcal T(T_{(n)})}\,C_{x'},\quad 
D_{\mathcal T} = \bigcup_{x'\in \mathcal T(T_{(n)})}\,D_{x'}.
\]
By Lemma~\ref{l:Lambda-n}, if $l\geq 10$, then the sets $D_{x'}$ in the above union are pairwise disjoint.
\begin{lemma}\label{l:D-C-hitting}
There exists $C_{\ref{l:D-C-hitting}}=C_{\ref{l:D-C-hitting}}(d)$ such that for all $n\geq 1$, $\mathcal T\in\Lambda_n$ and $l\geq C_{\ref{l:D-C-hitting}}$, 
\[
\sup_{y\in D_{\mathcal T}}\,\prw_y\left[H_{C_{\mathcal T}}<\infty\right]\leq \frac{1}{2e}.
\]
\end{lemma}
\begin{proof}[Proof of Lemma~\ref{l:D-C-hitting}]
Let $n\geq 1$, $\mathcal T\in\Lambda_n$ and $y\in D_{\mathcal T}$.

Denote by $S_k$ the set of all $x'\in \mathcal T(T_{(n)})$ with $\frac{1}{5}L_k\leq \|x'-y\|\leq \frac15 L_{k+1}$. 
By Lemma~\ref{l:Lambda-n}(2), if $l\geq 10$, then $|S_k|\leq 2^k$. Also, $S_0=\emptyset$. Using \eqref{eq:hitting1}, we get 
\[
\prw_y\left[H_{C_{\mathcal T}}<\infty\right] 
\leq \sum_{k=1}^\infty\sum_{x'\in S_k}\prw_y\left[H_{C_{x'}}<\infty\right] 
\leq \sum_{k=1}^\infty\,|S_k|\,C\,L_0^{d-2}\,L_k^{2-d}
\leq C\,\sum_{k=1}^\infty (2l^{2-d})^k \leq \frac{1}{2e},
\]
for all $l$ sufficiently large.
\end{proof}

\smallskip

The next lemma is the main ingredient for the proof of Proposition~\ref{prop:bad-*paths}. 
\begin{lemma}\label{l:bad-leaves}
Let $d\geq 3$. 
For any $K\geq 1$, there exist $R_0=R_0(K)$ and $\alpha_0 = \alpha_0(K,R)>0$ such that 
for all $R\geq R_0$, $\alpha\leq\alpha_0$, $l\geq C_{\ref{l:D-C-hitting}}$, $n\geq 1$ and $\mathcal T\in\Lambda_n$,
\[
\P\left[\mathcal T(T_{(n)})\subseteq\bad(\mathcal N^\alpha)\right]\leq \exp\left(-K\,2^n\right).
\]
\end{lemma}
\begin{proof}[Proof of Lemma~\ref{l:bad-leaves}]
Let $n\geq 1$ and $\mathcal T\in\Lambda_n$. Take $l\geq C_{\ref{l:D-C-hitting}}$, $\alpha\leq 1$ and $M=K+2$. 

\smallskip

Recall that for two disjoint sets $A,B$, $\nume^\alpha_{A,B}$ denotes the number of excursions of all loops from $\mathscr L^\alpha$ from $A$ to $B$. 
Then, 
\begin{multline}\label{eq:bad-leaves-main}
\P\left[\mathcal T(T_{(n)})\subseteq\bad(\mathcal N^\alpha)\right]\\[5pt]
\leq 
\P\left[\nume^\alpha_{C_{\mathcal T},D_{\mathcal T}}\geq M\,2^n\right] + 
\P\left[\nume^\alpha_{C_{\mathcal T},D_{\mathcal T}}\leq M\,2^n,\, \mathcal T(T_{(n)})\subset\bad(\mathcal N^\alpha)\right].
\end{multline}
By the choice of $l$, Lemma~\ref{l:D-C-hitting} and Lemma~\ref{l:excursions-AB-num}, 
\begin{equation}\label{eq:bad-leaves-1}
\P\left[\nume^\alpha_{C_{\mathcal T},D_{\mathcal T}}\geq M\,2^n\right]\leq \exp\left(\alpha - M2^n\right)\leq \frac12\,\exp\left(-K\,2^n\right),
\end{equation}
where in the second inequality we used $\alpha\leq 1$ and $M=K+2$.

To bound the second term in \eqref{eq:bad-leaves-main}, recall that by the choice of $l$, the sets $D_{x'}$, $x'\in\mathcal T(T_{(n)})$, are pairwise disjoint. 
Thus, 
\[
\nume^\alpha_{C_{\mathcal T},D_{\mathcal T}} = \sum_{x\in\mathcal T(T_{(n)})}\nume^\alpha_{C_{x'},D_{x'}}\,.
\]
In particular, if $\nume^\alpha_{C_{\mathcal T},D_{\mathcal T}}\leq M\,2^n$, then there exists a subset $S$ of $\mathcal T(T_{(n)})$ with cardinality $2^{n-1}$ such that 
$\nume^\alpha_{C_{x'},D_{x'}} \leq 2M$ \emph{for all} $x'\in S$. 
As the number of possible subsets of $\mathcal T(T_{(n)})$ with cardinality $2^{n-1}$ is at most $2^{2^n}$, we obtain that
\begin{multline*}
\P\left[\nume^\alpha_{C_{\mathcal T},D_{\mathcal T}}\leq M\,2^n,\, \mathcal T(T_{(n)})\subset\bad(\mathcal N^\alpha)\right]\\[5pt]
\leq 2^{2^n}\,\sup_S 
\P\left[\nume^\alpha_{C_{x'},D_{x'}}\leq 2M\text{ and }x'\in\bad(\mathcal N^\alpha)\text{ for all $x'\in S$}\right],
\end{multline*}
where the supremum is over all subsets $S$ of $\mathcal T(T_{(n)})$ with cardinality $2^{n-1}$.

\smallskip

The event that $x'$ is $R$-bad only depends on the restriction of $\mathcal N^\alpha$ to $\cube(x')$. 
Thus, if we denote by $\mathcal N^\alpha_{x'}$ the total local time of all loops from $\mathscr L^\alpha$ that intersect $\cube(x')$ but not $D_{x'}$, 
then for all $z\in \cube(x')$, $\mathcal N^\alpha(z)$ is the sum of $\mathcal N^\alpha_{x'}(z)$ and 
the total number of visits to $z$ of all the excursions of $\mathscr L^\alpha$ from $C_{x'}$ to $D_{x'}$. 
Note that 
\begin{itemize}\itemsep0pt
\item
$\mathcal N^\alpha_{x'}$, $x'\in S$, are independent,
\item
the excursions of $\mathscr L^\alpha$ from $C_{x'}$ to $D_{x'}$, conditioned on their starting and ending locations, 
are distributed as independent random walk bridges (see Proposition~\ref{prop:sampling-loopsoup}),
\item
the event that $x'$ is $R$-bad for $\lt:\Z^d\to\N$ is increasing in $\lt$.
\end{itemize}
Thus, if we denote by $\mathcal N'$ the total local time of $2M$ random walk excursions from $C_0$ to $D_0$, then 
\begin{multline*}
\P\left[\nume^\alpha_{C_{x'},D_{x'}}\leq 2M\text{ and }x'\in\bad(\mathcal N^\alpha)\text{ for all $x'\in S$}\right]\\
\leq
\left(\max_{(y_i,z_i)_{i=1}^{2M}}\,\P\otimes\bigotimes_{i=1}^{2M}\prw_{y_i,z_i}^{D_0}\left[\text{$0$ is $R$-bad for $\left(\mathcal N^\alpha + \mathcal N'\right)$}\right]\right)^{2^{n-1}}\\
\leq 
\left(\P\left[\sum_{z\in Q(0)}\mathcal N^\alpha(z)\geq 1\right] + 
\max_{(y_i,z_i)_{i=1}^{2M}}\,\bigotimes_{i=1}^{2M}\prw_{y_i,z_i}^{D_0}\left[\text{$0$ is $R$-bad for $\mathcal N'$}\right]\right)^{2^{n-1}},
\end{multline*}
where the maximum is over all $2M$-tuples of pairs $(y_i,z_i)\in C_0\times D_0$---the starting and ending locations of excursions from $C_0$ to $D_0$. 

\smallskip

It remains to prove that for a suitable choice of $\alpha$ and $R$, 
\begin{equation}\label{eq:bad-leaves-2}
\P\left[\sum_{z\in Q(0)}\mathcal N^\alpha(z)\geq 1\right] \leq \frac{1}{16}\,\exp\left(-2K\right)
\end{equation}
and 
\begin{equation}\label{eq:bad-leaves-3}
\max_{(y_i,z_i)_{i=1}^{2M}}\,\bigotimes_{i=1}^{2M}\prw_{y_i,z_i}^{D_0}\left[\text{$0$ is $R$-bad for $\mathcal N'$}\right]\leq \frac{1}{16}\,\exp\left(-2K\right).
\end{equation}
Indeed, if \eqref{eq:bad-leaves-2} and \eqref{eq:bad-leaves-3} hold, then the second summand in \eqref{eq:bad-leaves-main} is bounded from above by 
\[
2^{2^n}\,\left(\frac{1}{8}\,\exp\left(-2K\right)\right)^{2^{n-1}}\leq \frac12\,\exp\left(-K2^n\right)
\]
and, combined with \eqref{eq:bad-leaves-1}, this gives the result. 

\smallskip

We begin with \eqref{eq:bad-leaves-3}. Let $(y_i,z_i)_{i=1}^{2M}$ be the $2M$-tuple for which the maximum is attained. 
By the definition of $R$-bad vertex, 
the probability in \eqref{eq:bad-leaves-3} is bounded from above by 
\begin{equation}\label{eq:Rd-1}
\sum_{i=1}^{2M}\prw_{y_i,z_i}^{D_0}\left[H_{\edges}<\infty\right] 
+ \sum_{i=1}^{2M}\prw_{y_i,z_i}^{D_0}\left[\sum_n\sum_{x\in\dint\cube(0)}\mathds{1}_{X_n = x}> \frac{1}{2M}R^{d-1}\right],
\end{equation}
which can be estimated using standard results for random walks and the fact that for any $d\geq 3$, there exists $C<\infty$ such that
\begin{equation}\label{eq:capaedges}
\capa(\edges) \leq \frac{C\, R^{d-2}}{\log R},\qquad R\geq 2\,,\quad (\text{cf. \cite[Lemma~3.2]{DRS-AIHP}}).  
\end{equation}
Indeed, by \eqref{eq:capaedges}, \eqref{eq:H-g-e}, \eqref{eq:GF} and the Harnack principle, the first sum is bounded from above by $\frac{CM}{\log R}$. 
By the Markov inequality, \eqref{eq:GF} and the Harnack principle, the second sum is bounded from above by $CM^2R^{2-d}$. 
Thus, if $R\geq R_0=R_0(K)$, then \eqref{eq:bad-leaves-3} holds. 

\smallskip

It remains to show that for $\alpha\leq \alpha_0 = \alpha_0(K,R)$, \eqref{eq:bad-leaves-2} holds, but this is immediate, 
since by properties of $\mathscr L^\alpha$, the probability in \eqref{eq:bad-leaves-2} is bounded from above by $CR^d\alpha$. 
\end{proof}

\bigskip

\begin{proof}[Proof of Proposition~\ref{prop:bad-*paths}]
First note that it suffices to prove that for some $R\geq 1$, $l\geq 1$ and $\alpha>0$, 
\begin{equation}\label{eq:bad-*paths-1}
\P\left[
\text{$\ball(0,L_n)$ is $*$-connected to $\dint \ball(0,2L_n)$ in $\bad(\ltloop^\alpha)$}
\right]
\leq 2^{-2^n} 
\end{equation}
for all $n\geq 1$. Indeed, let $N\geq 1$ and choose $n$ so that $2L_n\leq L_0N\leq 2L_{n+1}$. 
Then, the event in \eqref{eq:bad-*paths} implies the event in \eqref{eq:bad-*paths-1} and $N\leq \frac{2L_{n+1}}{L_0}=2l^{n+1}\leq 2^{Cn}$ for some $C=C(l)$.

\smallskip

Claim \eqref{eq:bad-*paths-1} easily follows from Lemma~\ref{l:bad-leaves} and the observation that 
the event in \eqref{eq:bad-*paths-1} implies the existence of an embedding $\mathcal T\in \Lambda_n$ 
such that the images of all leaves $T_{(n)}$ are $R$-bad for $\ltloop^\alpha$ 
(see, e.g., \cite[(3.24)]{Sznitman:Decoupling} or \cite[Lemma~3.3]{Rath}). Namely, 
\begin{multline*}
\P\left[
\text{$\ball(0,L_n)$ is $*$-connected to $\dint \ball(0,2L_n)$ in $\bad(\ltloop^\alpha)$}
\right]\\[5pt]
\leq 
\P\left[\text{there exists $\mathcal T\in\Lambda_n$ such that $\mathcal T(T_{(n)})\subset\bad(\ltloop^\alpha)$}\right]\\[5pt]
\stackrel{\mathrm{L.}\ref{l:Lambda-n}(1)}\leq 
\left((2d)^2\,(4l)^{2(d-1)}\right)^{2^n-1}\,
\sup_{\mathcal T\in \Lambda_n} \P\left[\mathcal T(T_{(n)})\subset\bad(\ltloop^\alpha)\right].
\end{multline*}
Let $l\geq C_{\ref{l:D-C-hitting}}$ and choose $K=K(l)$ so that 
\[
\left((2d)^2\,(4l)^{2(d-1)}\right)^{2^n-1}\,\exp\left(-K2^n\right)\leq 2^{-2^n}.
\]
Finally, choose $R=R_0(K)$ and $\alpha = \alpha_0(R,K)>0$ as in Lemma~\ref{l:bad-leaves}. 
Then, by Lemma~\ref{l:bad-leaves}, 
\[
\sup_{\mathcal T\in \Lambda_n} \P\left[\mathcal T(T_{(n)})\subset\bad(\ltloop^\alpha)\right]\leq \exp\left(-K2^n\right),
\]
and \eqref{eq:bad-*paths-1} follows for this choice of $l$, $R$ and $\alpha$. 
\end{proof}

\subsection{Proof of Lemma~\ref{l:surgery}}\label{sec:surgery}

We begin with an outline of the proof. For $x'\in \GG_0$, we decompose all the loops from the loop soup $\mathscr L^\alpha$ that visit $A=\dint\cube (x')$ and $B=\dext\cube(x')$ 
into inner (from $A$ to $B$) and outer (from $B$ to $A$) excursions. 
By Proposition~\ref{prop:sampling-loopsoup}, given their starting and ending locations, the inner and outer excursions are independent random walk bridges. 
In view of independence, the conditional probability in \eqref{eq:surgery} with respect to the $\sigma$-algebras generated by all good boxes and all the vacant set in the complement of $\cube(x')$
can be substituted by the conditional probability with respect to only the starting and ending locations of the inner excursions and the event that $x'$ is good, cf.\ \eqref{eq:surgery:1} and \eqref{eq:surgery:2}. 
Now, by Definition~\ref{def:good}(2) of the good box (see also Remark~\ref{rem:good}) the total number of inner excursions is bounded from above by $R^{d-1}$. 
Since all of them are distributed as independent random walk bridges, one can prescribe their values as simple paths inside of $\cube(x')$ in such a way that 
a given point $y\in\dext \cube(x')$ is connected to $\edges(x')$ by a nearest neighbor path in $\cube(x')$ which is avoided by all the bridges, see \eqref{eq:rho} and below. 
Since the number of bridges is bounded and each is realized as a simple path in $\cube(x')$, the price of such a local surgery is uniformly positive. 
Furthermore, with positive probability there are no loops of $\mathscr L^\alpha$ that are entirely contained in $\cube(x')$, thus the constructed nearest neighbor path from $y$ to $\edges(x')$ 
in $\cube(x')$ is in fact a path in the vacant set $\mathcal V^\alpha$.
Finally, such a surgery keeps $x'$ good.

\medskip

We proceed with the details of the proof. Let $x'\in \GG_0$ and $y\in\dext \cube(x')$. 
Define 
\[
A=\dint\cube (x'),\quad B=\dext\cube(x'),
\]
and recall from \eqref{def:mathcalE} the definition of Poisson point processes 
$\mathcal E^{\alpha,j}_{A,B}$, $\overrightarrow{\mathcal E}^{\alpha,j}_{A,B}$, and $\overleftarrow{\mathcal E}^{\alpha,j}_{A,B}$, $j\geq 1$, 
of pairs of loop entrance points in $A$ and $B$, inner and outer bridges, respectively. 
Define sigma-algebras 
\[
\mathcal E = \sigma\left(\mathcal E^{\alpha,j}_{A,B},\,j\geq 1\right),\quad
\overrightarrow{\mathcal E} = \sigma\left(\overrightarrow{\mathcal E}^{\alpha,j}_{A,B},\,j\geq 1\right),\quad
\overleftarrow{\mathcal E} = \sigma\left(\overleftarrow{\mathcal E}^{\alpha,j}_{A,B},\,j\geq 1\right),
\]
and the sigma-algebra $\mathcal F_{\mathrm{ext}}$ generated by the loops from $\mathscr L^\alpha$ that do not intersect $\cube(x')$.

\medskip

Let $x$ be the unique neighbor of $y$ in $\dint\cube(x')$ and consider the event $D$ that 
$x$ is connected to $\edges(x')$ in $\mathcal V^\alpha\cap\cube(x')$. Then, 
\[
\left\{\text{$y$ is connected to $\edges(x')$ in $\mathcal V^\alpha\cap(\{y\}\cup\cube(x'))$}\right\} = D\cap\{y\in\mathcal V^\alpha\}.
\]
Finally, let $\mathcal E(\check x,\check y)$ be the event that none of the loop excursions from $A$ to $B$ starts at $x$ and none of them ends at $y$, 
namely, for all the pairs of points in $\mathcal E^{\alpha, j}_{A,B}$, $j\geq 1$, the first point is not $x$ and the second is not $y$. 
Note that $\{y\in\mathcal V^\alpha\}\subseteq \mathcal E(\check x,\check y)$.

\medskip

To prove \eqref{eq:surgery} it suffices to show that 
\begin{equation}\label{eq:surgery:1}
\mathbb P\left[D\,\Big|\,\sigma(\mathds{1}_{x'\in\good(\ltloop^\alpha)}),\,\overleftarrow{\mathcal E},\,\mathcal F_{\mathrm{ext}} \right]
\geq \gamma\,\mathds{1}_{\mathcal E(\check x,\check y), x'\in\good(\ltloop^\alpha)},\,\,\,\, \mathbb P\text{-a.s.}
\end{equation}
Indeed, 
\begin{multline*}
\mathbb P\left[
\text{$y$ is connected to $\edges(x')$ in $\mathcal V^\alpha\cap(\{y\}\cup\cube(x'))$, $x'\in\good(\ltloop^\alpha)$}
\,\Big|\,\Sigma_{\mathcal G},\,\mathcal A_{x'} \right] \\[7pt]
\begin{array}{l}
= 
\mathbb P\left[
D,\,y\in\mathcal V^\alpha,\, x'\in\good(\ltloop^\alpha)
\,\Big|\,\Sigma_{\mathcal G},\,\mathcal A_{x'} \right] \\[7pt]
= 
\mathbb E\left[
\mathbb P\left[
D,\,y\in\mathcal V^\alpha,\, x'\in\good(\ltloop^\alpha)
\,\Big|\,\sigma(\mathds{1}_{x'\in\good(\ltloop^\alpha)}),\,\overleftarrow{\mathcal E},\,\mathcal F_{\mathrm{ext}}\right]
\,\Big|\,\Sigma_{\mathcal G},\,\mathcal A_{x'}\right]\\[7pt]
=
\mathds{1}_{y\in\mathcal V^\alpha, x'\in\good(\ltloop^\alpha)}\,
\mathbb E\left[
\mathbb P\left[D\,\Big|\,\sigma(\mathds{1}_{x'\in\good(\ltloop^\alpha)}),\,\overleftarrow{\mathcal E},\,\mathcal F_{\mathrm{ext}}\right]
\,\Big|\,\Sigma_{\mathcal G},\,\mathcal A_{x'}\right]\\[7pt]
\stackrel{\eqref{eq:surgery:1}}\geq 
\gamma\,\mathds{1}_{y\in\mathcal V^\alpha, x'\in\good(\ltloop^\alpha)}\,
\mathbb E\left[\mathds{1}_{\mathcal E(\check x,\check y), x'\in\good(\ltloop^\alpha)}
\,\Big|\,\Sigma_{\mathcal G},\,\mathcal A_{x'}\right]\\[7pt]
\geq \gamma\,\mathds{1}_{y\in\mathcal V^\alpha, x'\in\good(\ltloop^\alpha)},\qquad\text{which gives \eqref{eq:surgery}.}
\end{array}
\end{multline*}

\medskip

By the definition of Poisson point process, the sigma-algebras $\mathcal F_{\mathrm{ext}}$ and 
$\sigma(\mathcal E, \overrightarrow{\mathcal E}, \overleftarrow{\mathcal E})$ are independent. Furthermore, 
by Proposition~\ref{prop:sampling-loopsoup}, the sigma-algebras $\overrightarrow{\mathcal E}$ and $\overleftarrow{\mathcal E}$ are conditionally independent given $\mathcal E$. 
Thus, 
\begin{equation}\label{eq:surgery:2}
\mathbb P\left[D\,\Big|\,\sigma(\mathds{1}_{x'\in\good(\ltloop^\alpha)}),\,\overleftarrow{\mathcal E},\,\mathcal F_{\mathrm{ext}} \right]
= 
\mathbb P\left[D\,\Big|\,\sigma(\mathds{1}_{x'\in\good(\ltloop^\alpha)}),\,\mathcal E \right],\,\,\,\, \mathbb P\text{-a.s.}
\end{equation}
Indeed, by Dynkin's $\pi$-$\lambda$ lemma, it suffices to show that for any admissible $\mathrm{e}$, $\overleftarrow{\mathrm e}$, and $F\in\mathcal F_{\mathrm{ext}}$, 
\begin{multline*}
\mathbb P\left[
D,\,x'\in\good(\ltloop^\alpha),\,\{\mathcal E^{\alpha,j}_{A,B}\}_{j\geq 1} = \mathrm{e},\,\{\overleftarrow{\mathcal E}^{\alpha,j}_{A,B}\}_{j\geq 1} = \overleftarrow{\mathrm{e}},\,F\right]\\
=
\mathbb E\left[x'\in\good(\ltloop^\alpha),\,\{\mathcal E^{\alpha,j}_{A,B}\}_{j\geq 1} = \mathrm{e},\,\{\overleftarrow{\mathcal E}^{\alpha,j}_{A,B}\}_{j\geq 1} = \overleftarrow{\mathrm{e}},\,F,\,
\mathbb P\left[D\,\Big|\,\sigma(\mathds{1}_{x'\in\good(\ltloop^\alpha)}),\,\mathcal E\right]
\right],
\end{multline*}
which is immediate, since by the (conditional) independence of sigma-algebras,
\begin{multline*}
\mathbb P\left[
D,\,x'\in\good(\ltloop^\alpha),\,\{\mathcal E^{\alpha,j}_{A,B}\}_{j\geq 1} = \mathrm{e},\,\{\overleftarrow{\mathcal E}^{\alpha,j}_{A,B}\}_{j\geq 1} = \overleftarrow{\mathrm{e}},\,F\right]\\
= 
\mathbb P\left[D,\,x'\in\good(\ltloop^\alpha),\,\{\mathcal E^{\alpha,j}_{A,B}\}_{j\geq 1} = \mathrm{e}\right]
\,
\mathbb P\left[\{\overleftarrow{\mathcal E}^{\alpha,j}_{A,B}\}_{j\geq 1} = \overleftarrow{\mathrm{e}},\,F\,\Big|\,\{\mathcal E^{\alpha,j}_{A,B}\}_{j\geq 1} = \mathrm{e}\right]\\
= 
\mathbb E\left[x'\in\good(\ltloop^\alpha),\,\{\mathcal E^{\alpha,j}_{A,B}\}_{j\geq 1} = \mathrm{e},\,
\mathbb P\left[D\,\Big|\,\sigma(\mathds{1}_{x'\in\good(\ltloop^\alpha)}),\,\mathcal E\right]
\right]\\
\mathbb P\left[\{\overleftarrow{\mathcal E}^{\alpha,j}_{A,B}\}_{j\geq 1} = \overleftarrow{\mathrm{e}},\,F\,\Big|\,\{\mathcal E^{\alpha,j}_{A,B}\}_{j\geq 1} = \mathrm{e}\right]\\
=
\mathbb E\left[x'\in\good(\ltloop^\alpha),\,\{\mathcal E^{\alpha,j}_{A,B}\}_{j\geq 1} = \mathrm{e},\,\{\overleftarrow{\mathcal E}^{\alpha,j}_{A,B}\}_{j\geq 1} = \overleftarrow{\mathrm{e}},\,F,\,
\mathbb P\left[D\,\Big|\,\sigma(\mathds{1}_{x'\in\good(\ltloop^\alpha)}),\,\mathcal E\right]
\right],
\end{multline*}
for all compatible $\mathrm{e}$ and $\overleftarrow{\mathrm{e}}$. 

\medskip

Thus, by \eqref{eq:surgery:1} and \eqref{eq:surgery:2}, it suffices to prove that 
\[
\mathbb P\left[D\,\Big|\,\sigma(\mathds{1}_{x'\in\good(\ltloop^\alpha)}),\,\mathcal E \right]
\geq \gamma\,\mathds{1}_{\mathcal E(\check x,\check y), x'\in\good(\ltloop^\alpha)},\,\,\,\, \mathbb P\text{-a.s.},
\]
in other words, that for all $\mathrm{e}$ such that $\{\{\mathcal E^{\alpha,j}_{A,B}\}_{j\geq 1} = \mathrm{e}\}\subseteq \mathcal E(\check x,\check y)$, 
\begin{equation}\label{eq:surgery:3}
\mathbb P\left[D,\,x'\in\good(\ltloop^\alpha),\,\{\mathcal E^{\alpha,j}_{A,B}\}_{j\geq 1} = \mathrm{e}\right]
\geq \gamma\,\mathbb P\left[x'\in\good(\ltloop^\alpha),\,\{\mathcal E^{\alpha,j}_{A,B}\}_{j\geq 1} = \mathrm{e}\right].
\end{equation}
In particular, we may and will assume from now on that 
\[
x\notin\edges(x'),
\]
since otherwise the claim is trivial. 

In fact, we will show a stronger statement. 
Let $F_{\mathrm{int},\emptyset}$ be the event that the set of loops from $\mathscr L^\alpha$ contained in $\cube(x')$ is empty, then 
\begin{equation}\label{eq:surgery:4}
\mathbb P\left[D,\,x'\in\good(\ltloop^\alpha),\,\{\mathcal E^{\alpha,j}_{A,B}\}_{j\geq 1} = \mathrm{e},\,F_{\mathrm{int},\emptyset}\right]
\geq \gamma\,\mathbb P\left[\{\mathcal E^{\alpha,j}_{A,B}\}_{j\geq 1} = \mathrm{e}\right]
\end{equation}
for all $\mathrm{e}$ as in \eqref{eq:surgery:3} and satisfying additionally $\{\{\mathcal E^{\alpha,j}_{A,B}\}_{j\geq 1} = \mathrm{e}\}\cap\{x'\in\good(\ltloop^\alpha)\}\neq\emptyset$. 
(This basically means that none of the loop excursions can start from $\edges(x')$ or end in a neighbor of $\edges(x')$ and that the total number of excursions 
does not exceed $\frac12 R^{d-1}$, cf.\ Definition~\ref{def:good}.)

\smallskip

Let $\overrightarrow{\mathcal N}^\alpha$ be the field of cumulative occupation local times in $\cube(x')$ of all the excursions from $\{\overrightarrow{\mathcal E}^{\alpha,j}_{A,B}\}_{j\geq 1}$, 
that is, for $z\in\cube(x')$, $\overrightarrow{\mathcal N}^\alpha(z)$ is the total number of times $z$ is visited by the bridges $\{\overrightarrow{\mathcal E}^{\alpha,j}_{A,B}\}_{j\geq 1}$. 
Also, let $\overrightarrow{\mathcal V}^\alpha = \{z\in\cube(x'):\overrightarrow{\mathcal N}^\alpha(z) = 0\}$. 
Note that 
\[
\{D,\,x'\in\good(\ltloop^\alpha),\,F_{\mathrm{int},\emptyset}\} = 
\{\text{$x$ is connected to $\edges(x')$ in $\overrightarrow{\mathcal V}^\alpha$},\,x'\in\good(\overrightarrow{\mathcal N}^{\alpha})\}\cap F_{\mathrm{int},\emptyset},
\]
and the two events on the right are independent. 
Since the number of loops from $\mathscr L^\alpha$ contained in $\cube(x')$ is a Poisson random variable with parameter $\alpha c$, for $c=c(R)$, 
\[
\mathbb P[F_{\mathrm{int},\emptyset}] = e^{-\alpha c}\geq e^{-\overline\alpha c}>0,
\]
and to finish the proof of \eqref{eq:surgery:4} it suffices to show that for all $\mathrm{e}$ as before and some $\gamma_1=\gamma_1(d,R)>0$, 
\[
\mathbb P\left[\text{$x$ is connected to $\edges(x')$ in $\overrightarrow{\mathcal V}^\alpha$},\,
x'\in\good(\overrightarrow{\mathcal N}^{\alpha}),\,\{\mathcal E^{\alpha,j}_{A,B}\}_{j\geq 1} = \mathrm{e}\right]
\geq \gamma_1\,\mathbb P\left[\{\mathcal E^{\alpha,j}_{A,B}\}_{j\geq 1} = \mathrm{e}\right],
\]
or, equivalently, that 
\begin{equation}\label{eq:surgery:5}
\mathbb P\left[\text{$x$ is connected to $\edges(x')$ in $\overrightarrow{\mathcal V}^\alpha$},\,
x'\in\good(\overrightarrow{\mathcal N}^{\alpha})\,\Big|\,\{\mathcal E^{\alpha,j}_{A,B}\}_{j\geq 1} = \mathrm{e}\right]
\geq \gamma_1.
\end{equation}
Let $\mathrm{e} = \{(x_i,y_i)\in A\times B,\,1\leq i\leq N\}$ be 
a multiset of all starting and ending locations of all the excursions of loops from $\mathscr L^\alpha$ from $A$ to $B$, 
which satisfies all the above assumptions on $\mathrm{e}$. By Proposition~\ref{prop:sampling-loopsoup}, the law of the excursions 
$\{\overrightarrow{\mathcal E}^{\alpha,j}_{A,B}\}_{j\geq 1}$, conditioned on $\{\mathcal E^{\alpha,j}_{A,B}\}_{j\geq 1} = \mathrm{e}$, 
is the law of independent random walk bridges from $x_i$ conditioned to enter $B$ in $y_i$, that is $\bigotimes_{i=1}^N \prw_{x_i,y_i}^B$.

Let $\{\mathcal X_i\}_{i=1}^N$ be a family of independent random walk bridges distributed according to $\bigotimes_{i=1}^N \prw_{x_i,y_i}^B$. 
Let $\overrightarrow{\mathcal N}$ be the field of cumulative occupation local times in $\cube(x')$ of all the bridges $\mathcal X_i$, 
that is, for $z\in\cube(x')$, $\overrightarrow{\mathcal N}(z)$ is the total number of times $z$ is visited by the bridges $\mathcal X_i$. 
Also, let $\overrightarrow{\mathcal V} = \{z\in\cube(x'):\overrightarrow{\mathcal N}(z) = 0\}$. Then, \eqref{eq:surgery:5} is equivalent to 
\begin{equation}\label{eq:surgery:6}
\bigotimes_{i=1}^N \prw_{x_i,y_i}^B \left[\text{$x$ is connected to $\edges(x')$ in $\overrightarrow{\mathcal V}$},\,
x'\in\good(\overrightarrow{\mathcal N})\right] \geq \gamma_1,
\end{equation}
for any choice of $\{(x_i,y_i)\in A\times B,\,1\leq i\leq N\}$ such that $N\leq \frac12 R^{d-1}$ and for all $i$, $x_i\notin\{x\}\cup\edges(x')$ and $y_i\notin\{y\}\cup\dext\edges(x')$.

\medskip

We prove that there exist $N$ simple (deterministic) paths $\rho_i$ from $x_i$ to $y_i$, such that 
\begin{equation}\label{eq:rho}
\begin{array}{l}
\text{(a) $\bigotimes_{i=1}^N \prw_{x_i,y_i}^B \left[\mathcal X_i = \rho_i,\,1\leq i\leq N\right] \geq \gamma_1$ and}\\[3pt]
\text{(b) event $\{\mathcal X_i = \rho_i,\,1\leq i\leq N\}$ implies the event inside probability in \eqref{eq:surgery:6}.}
\end{array}
\end{equation}
Once the existence of such paths $\rho_i$ is shown, \eqref{eq:surgery:6} is immediate. 

\smallskip

Recall that we assume $x\notin\edges(x')$. Thus, precisely one of the coordinates, say coordinate $i$, of the vector $x - x'$ is $-R$ or $R$, and 
the other coordinates take values between $-R+3$ and $R-3$. 
Let $j$ be the first coordinate which is not equal to $i$  
and denote by $e_s$ the $s$th coordinate unit vector. We define the set $\Pi$ in $\cube(x')$ as
\[
\{x, x + e_i, x + 2e_i\}\cup(\{x + 2e_i + te_j~:~t\geq 0\}\cap\cube(x'))
\]
if the $i$th coordinate of $x-x'$ equals $-R$, and as 
\[
\{x, x - e_i, x - 2e_i\}\cup(\{x - 2e_i + te_j~:~t\geq 0\}\cap\cube(x'))
\]
if the $i$th coordinate of $x-x'$ equals $R$, see Figure~\ref{fig:surgery}. 
Note that for $R\geq 4$, 
\begin{itemize}
\item
$\Pi\cap\edges(x') \neq \emptyset$,
\item
$\overline\cube = \cube(x')\setminus \left(\dint\cube(x')\cup\edges(x')\cup\Pi\right)$ is a connected subset of $\cube(x')$,
\item
every $z\in\dint\cube(x')\setminus\left(\edges(x')\cup\{x\}\right)$ has a neighbor in $\overline\cube$.
\end{itemize}
\begin{figure}[!tp]
\centering
\resizebox{15cm}{!}{\includegraphics{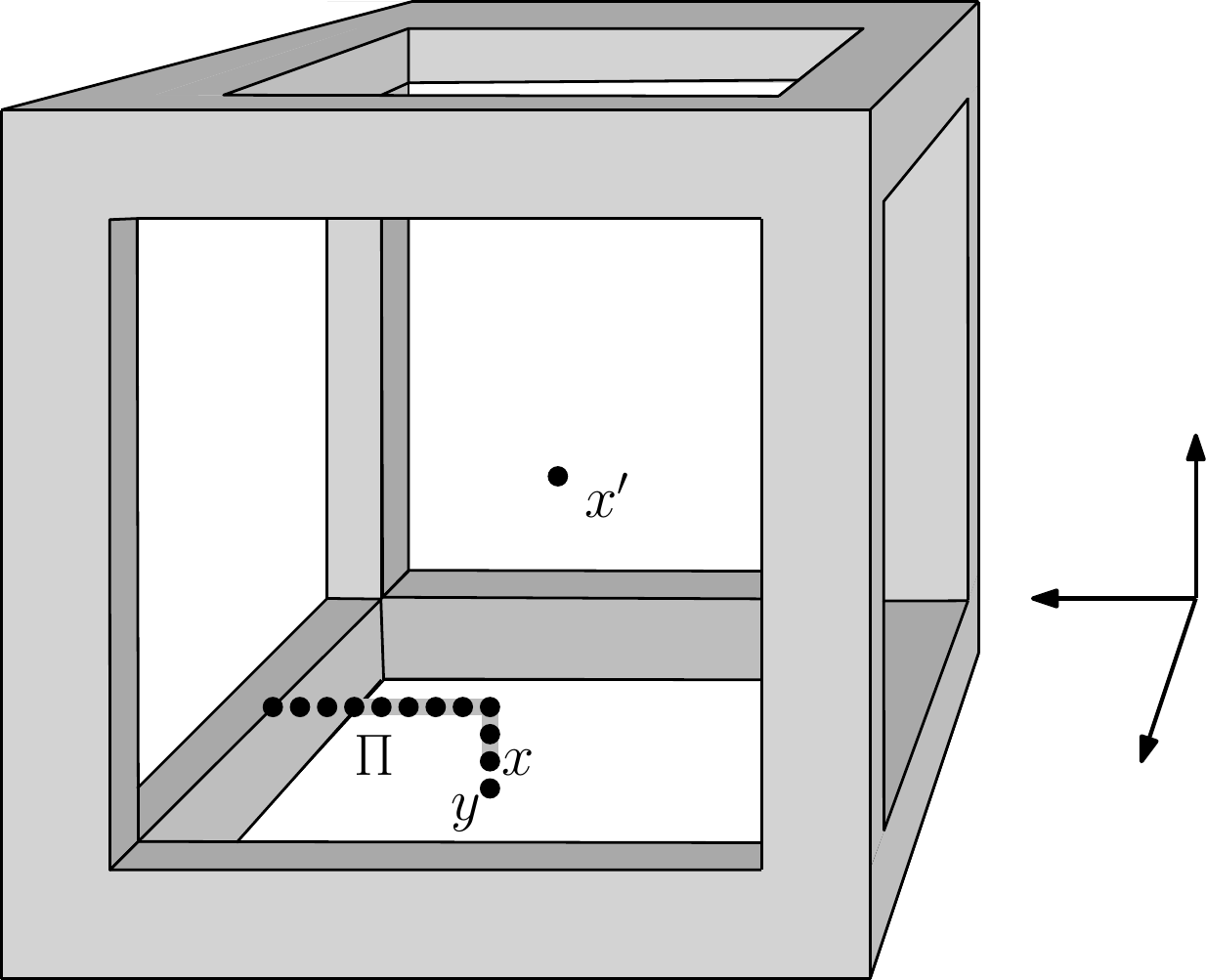}\hspace{4em}\includegraphics{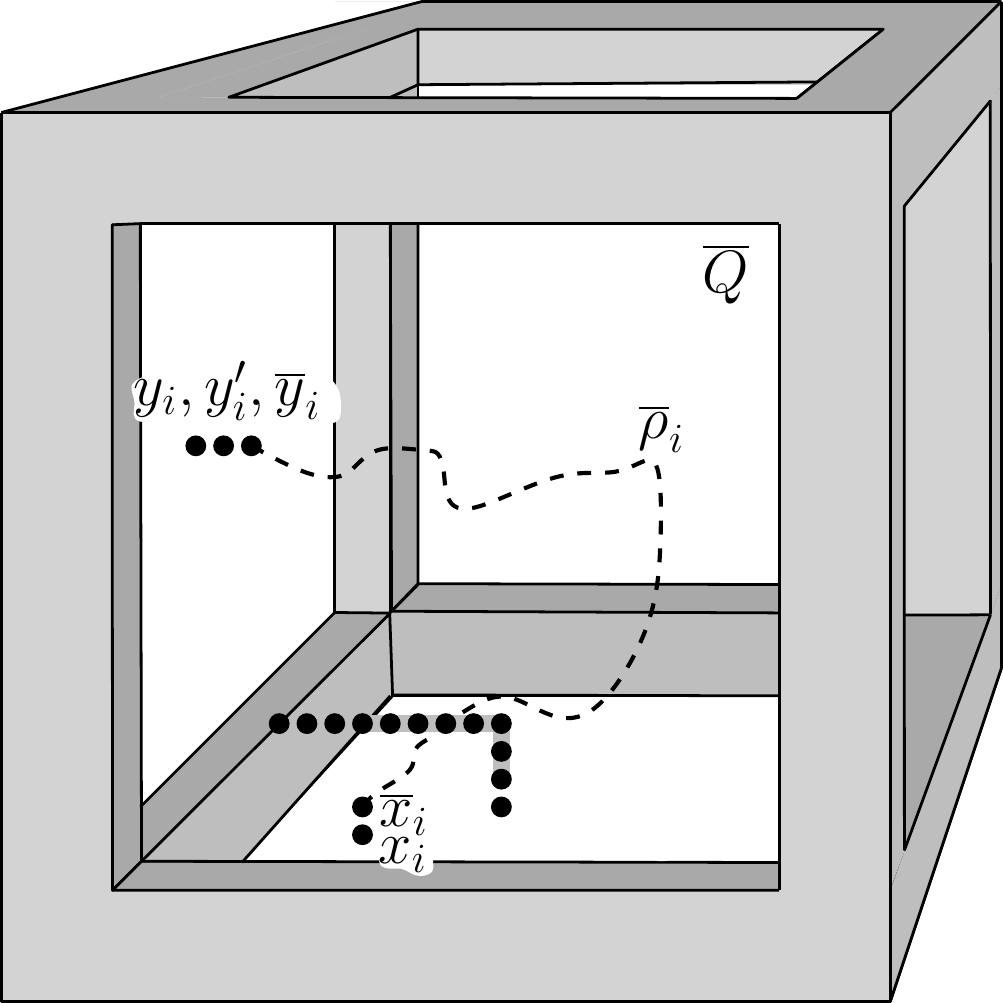}}
\caption{On the left, the ``tunnel'' $\Pi$, which connects $x$ to $\edges(x')$ inside of $\cube(x')$. 
On the right, a simple path $\overline{\rho}_i$ between $\overline{x}_i$ and $\overline{y}_i$ 
inside the connected set $\overline\cube = \cube(x')\setminus \left(\dint\cube(x')\cup\edges(x')\cup\Pi\right)$. 
The simple path $\rho_i$, defined as $(x_i,\overline{\rho}_i,y_i',y_i)$, visits the boundary $\dint\cube(x')$ exactly $2$ times, namely, at $x_i$ and $y_i'$.} 
\label{fig:surgery}
\end{figure}
Coming back to the random walk bridges, for each $x_i$ and $y_i$, 
let $\overline{x}_i$ be the unique neighbor of $x_i$ in $\overline\cube$ (note that $x_i\in \dint\cube(x')\setminus\left(\edges(x')\cup\{x\}\right)$ by assumptions), 
$y_i'$ the unique neighbor of $y_i$ in $\cube(x')$ (note that $y_i'\in \dint\cube(x')\setminus\left(\edges(x')\cup\{x\}\right)$) 
and $\overline{y}_i$ the unique neighbor of $y_i'$ in $\overline\cube$. 
Let $\overline{\rho}_i$ be an arbitrary simple path from $\overline{x}_i$ to $\overline{y}_i$ in $\overline\cube$, see Figure~\ref{fig:surgery}.

\smallskip

We define $\rho_i$ as the path $(x_i,\overline{\rho}_i,y_i',y_i)$. Then, 
each $\rho_i$ is a simple path from $x_i$ to $y_i$ that avoids $\Pi$, visits $\dint\cube(x')$ exactly twice and stops on entering $B$ (at $y_i$). 
Thus, 
\begin{itemize}
\item
for each $i$, $\prw_{x_i,y_i}^B\left[\mathcal X_i = \rho_i\right]\geq (2d)^{-|\rho_i|}\geq (2d)^{-|\cube(x')|} = c(d,R)$ and 
\item
the total number of visits of all $\rho_i$ to $\dint\cube(x')$ is not bigger than $R^{d-1}$.
\end{itemize}
In other words, the collection of paths $\rho_i$ satisfies the desired properties \eqref{eq:rho}. 

This way, the proof of \eqref{eq:surgery:6} (hence of Lemma~\ref{l:surgery}) is complete. 
\qed

\section{General approach to correlated percolation models}\label{sec:correlated-percolations}

For $d\geq 2$, let $\Omega=\{0,1\}^{\Z^d}$ and $\set = \set(\atom) = \{x\in\Z^d:\atom(x) = 1\}$ 
the subgraph of $\Z^d$ induced by $\atom\in\Omega$.
Let $\mathcal F$ be the sigma-algebra on $\Omega$ generated by the coordinate maps $\Psi_x$, $x\in\Z^d$, 
and let $\mathbb P^u$, $u\in(a,b)$,
be a \emph{family} of probability measures on $(\Omega,\mathcal F)$, for some (fixed) $0<a<b<\infty$. 

Under general assumptions on the family $\{\P^u\}_{u\in(a,b)}$ introduced in \cite{DRS12} it has been proven that 
for each $u\in(a,b)$, the random set $\set$ contains a unique infinite connected component $\set_\infty$, which on large scales ``looks like $\Z^d$'', 
for instance, for $\P^u$-almost every $\omega\in\Omega$, balls in $\set_\infty$ have asymptotic deterministic shape \cite{DRS12}, 
the simple random walk on $\set_\infty$ converges to a Brownian motion with a deterministic positive diffusion constant \cite{PRS}, 
its transition probabilities satisfy quenched Gaussian heat kernel bounds and the local CLT, etc.\ \cite{S14}.
These assumptions on $\{\P^u\}_{u\in(a,b)}$ are the following.

\begin{itemize}
\item[\p{}] {\it (Ergodicity)}
For each $u\in(a,b)$, every lattice shift is measure preserving and ergodic on $(\Omega,\mathcal F,\mathbb P^u)$.

\item[\pp{}] {\it (Monotonicity)}
For any $a<u<u'<b$ and increasing event $G\in\mathcal F$, 
$\mathbb P^u[G] \leq \mathbb P^{u'}[G]$.

\item[\ppp{}] {\it (Decoupling)}
There exist $\RP,\LP <\infty$ and $\epsP,\constP>0$ such that for any integers $L\geq \LP$ and $R\geq \RP$, 
if $a<\widehat u<u<b$ satisfy $u\geq \left(1 + R^{-\constP}\right)\, \widehat u$, $x_1,x_2\in\Z^d$ satisfy $\|x_1 - x_2\| \geq R L$, 
$A_1,A_2\in\sigma(\Psi_y,\,y\in \ball(x_i,10L))$ are increasing events and $B_1,B_2\in\sigma(\Psi_y,\,y\in \ball(x_i,10L))$ are decreasing, 
then 
\[
\mathbb P^{\widehat u}\left[A_1\cap A_2\right] \leq 
\mathbb P^u\left[A_1\right] \cdot
\mathbb P^u\left[A_2\right] 
+ \exp\left(- e^{(\log L)^\epsP} \right) ,\
\]
and
\[
\mathbb P^u\left[B_1\cap B_2\right] \leq 
\mathbb P^{\widehat u}\left[B_1\right] \cdot
\mathbb P^{\widehat u}\left[B_2\right] 
+ \exp\left(- e^{(\log L)^\epsP}\right) .
\]
\item[\s{}] {\it (Local uniqueness)}
For each $u\in(a,b)$, there exist $\constS>0$ and $\RS<\infty$ so that 
for all $R\geq \RS$, 
\[
\mathbb P^u\left[ \, 
\set_\infty\cap \ball(0,R) \neq \emptyset \,
\right]
\geq 
1 - \exp\left(-(\log R)^{1+\constS}\right) , 
\]
and
\[
\mathbb P^u\left[
\begin{array}{c}
\text{any two connected subsets of $\set\cap \ball(0,R)$ with}\\
\text{diameter $\geq \textstyle{\frac{R}{10}}$ are connected in $\set\cap \ball(0,2R)$}
\end{array}
\right]
\geq 1 - \exp\left(-(\log R)^{1+\constS}\right) . 
\]
\item[\sss{}] {\it (Continuity)}
Function $\eta(u)= \mathbb P^u\left[0\in\set_\infty\right]$ is positive and continuous on $(a,b)$. 
\end{itemize}

\medskip

While properties \p{} and \s{} are rather natural and have been extensively used in the analysis of supercritical percolation models, 
conditions \pp{}, \ppp{} and \sss{} represent the novelty of this framework and serve as a substitute to independence. 
(In fact, \pp{} easily follows from \ppp{} and is stated separately only for convenience.) 
They provide a connection between the measures $\P^u$ with different values of the parameter and 
serve \emph{only} to prove the likeliness of certain patterns in $\set_\infty$, cf.\ \cite[Remark~1.9(1)]{S14}. 
More precisely, if an increasing, resp.\ decreasing, (\emph{seed}) event is unlikely with respect to measure $\P^{u+\delta}$, resp.\ $\P^{u-\delta}$, 
then by applying \ppp{} recursively, one concludes that a family of $2^n$ translates of the event 
sufficiently spread out on $\Z^d$ in a certain hierarchical manner (\emph{cascading events}) occur with probability $\leq 2^{-2^n}$ 
with respect to measure $\P^u$, cf.\ \cite[Theorem~4.1]{DRS12}. 
Then, one uses \sss{} to show that the probabilities of suitable seed events (cf.\ \cite[Section~5]{DRS12}) with respect to 
measures $\P^{u+\delta}$, resp.\ $\P^{u-\delta}$, and $\P^u$ are close for small enough $\delta$, cf.\ \cite[Lemmas~5.2 and 5.4]{DRS12}.
In other words, one starts with a suitable increasing, resp.\ decreasing, seed event unlikely with respect to $\P^u$, 
concludes that it is also unlikely with respect to $\P^{u+\delta}$, resp.\ $\P^{u-\delta}$, for small $\delta>0$, and obtains that 
sufficiently spread out translates of the seed event are unlikely with respect to $\P^u$, but now with an explicit bound on the probability. 
All the other arguments in \cite{DRS12}, as well as in \cite{PRS,S14}, 
do not require comparison of probability laws with different parameters and go through for each fixed $u$ if $\P^u$ satisfies \p{} and \s{}. 

\medskip

In this section we prove in Theorem~\ref{thm:badseed:proba} that the result of \cite[Theorem~4.1]{DRS12} holds for families of probability measures $\P^u$ that 
satisfy condition \d{}, which is weaker than \ppp{}. As \ppp{} is only used in \cite{DRS12,PRS,S14} to derive \cite[Theorem~4.1]{DRS12}, 
all the results about geometric properties of $\set_\infty$ proved in \cite{DRS12,PRS,S14} 
hold for families of probability measures $\P^u$ that satisfy \p{}, \pp{}, \d{}, \s{}, \sss{}, see Corollary~\ref{cor:P1-S2:D}. 
This weakening is crucial in the study of the vacant set of the random walk loop soup, since 
it satisfies \d{}, but not \ppp{} (see Remarks~\ref{rem:d}(4) and \ref{rem:p3-loopsoup}).

\medskip

The family of probability measures $\P^u$, $u\in(a,b)$, satisfies condition \d{} if 

\begin{itemize}
\item[\d{}]
There exist constants $C,c$ and $\beta,\gamma,\zeta>0$ such that 
for all $L,s\geq 1$, $x_1,x_2\in\R^d$ with $\|x_1-x_2\|=sL$ and $a<u<u'<b$,
\begin{itemize}[leftmargin=*]
\item[(a)]
if $A_i\in\sigma(\Psi_y:y\in \ball(x_i,L))$ are increasing events, then 
\begin{equation}\label{eq:di:increasing}
\P^u\left[A_1\cap A_2\right] \leq \P^{u'}\left[A_1\right] \, \P^{u'}\left[A_2\right] + C\,\exp\left(-c\,\min\left\{(u'-u)^\beta\,s^\gamma,\,e^{(\log L)^\zeta}\right\}\right),
\end{equation}
\item[(b)]
if $B_i\in\sigma(\Psi_y:y\in \ball(x_i,L))$ are decreasing events, then 
\begin{equation}\label{eq:di:decreasing}
\P^{u'}\left[B_1\cap B_2\right] \leq \P^u\left[B_1\right] \, \P^u\left[B_2\right] + C\,\exp\left(-c\,\min\left\{(u'-u)^\beta\,s^\gamma,\,e^{(\log L)^\zeta}\right\}\right).
\end{equation}
\end{itemize}
\end{itemize}

\begin{remark}\label{rem:d}
\begin{enumerate}[leftmargin=*]\itemsep0pt
 \item
Note that inequalities \eqref{eq:di:increasing} and \eqref{eq:di:decreasing} are always valid if $(u'-u)^\beta\,s^\gamma\leq 1$,
thus condition \d{} would not change if one additionally assumes that $u'-u \geq s^{-\frac{\gamma}{\beta}}$. 
Now it is immedate that \d{} implies \ppp{} (take $s=R$, $\epsP = \zeta$, $\constP = \frac{\gamma}{\beta}$). 

\item
If inequalities \eqref{eq:di:increasing} and \eqref{eq:di:decreasing} hold only for $u'-u\geq s^{-\chi}$ for some $0<\chi< \frac{\gamma}{\beta}$, 
then they hold for all $u<u'$ with $(\beta,\gamma,\zeta)$ replaced by $(\beta' = \beta, \gamma'=\beta\chi, \zeta' = \zeta)$.

\item
In applications one uses \d{} to prove certain behavior of $\set_\infty$ under $\P^u$ for a fixed $u$ (see discussion before the definition of \d{}),
thus one only needs \d{} for $u'$s in a vicinity of $u$. In other words, one can assume that $b-a<1$. 
If so, inequalities \eqref{eq:di:increasing} and \eqref{eq:di:decreasing} get weaker by enlarging $\beta$ or diminishing $\gamma$. 
Thus, the reader should think of $\gamma$ being small and $\beta$ large. 
Incidentally, \d{} is satisfied by the random interlacements and the level sets of the Gaussian free field with $\gamma = d-2$ and $\beta = 2$, see, e.g., \cite{PT15,PR15}. 

\item
By Theorem~\ref{thm:decoupling}, condition \d{} is satisfied by the range of the loop soup $\mathscr L^\alpha$ 
with $\gamma = d-2$ and $\beta = \frac12$ (and any $\zeta>0$). 

\item
The key differences between \d{} and \ppp{} are that 
\begin{itemize}[leftmargin=*]\itemsep0pt
\item[(a)]
in models with polynomially decaying correlations (such as random interlacements, the Gaussian free field and the random walk loop soup), 
condition \d{} holds automatically if $s\leq \epsilon(\log L)^{\frac1\gamma}$; 
this way it is more natural than \ppp{}, since it only postulates decorrelation of local events occuring in large boxes when the boxes are far apart 
in comparison to their size,
\item[(b)]
the error term in \ppp{} improves by passing to higher scales $L$, while the one in \d{} is essentially invariant under rescaling of $L$. 
\end{itemize}
\end{enumerate}
\end{remark}

\begin{remark}\label{rem:p3-loopsoup}
The observations in Remark~\ref{rem:d}(5) are crucial for why \ppp{} is not a valid condition for the loop soup percolation.
Indeed, the range of the loop soup in disjoint boxes is correlated because of big loops that visit both boxes. 
If the boxes and the distance between them have the same scale (of order $L$, resp., $RL$ with a large but fixed $R$), 
then the stochastic behavior of the macroscopic loops visiting these boxes is essentially independent of the scale $L$. 
(Note that the loop soup on $\frac{1}{L}\Z^d$ converges for large $L$ to the Brownian loop soup, see, e.g., \cite{SS-LEW}.)
Using this observation, Chang proved in \cite{Ch15} that condition \ppp{} does not hold 
for events 
\begin{eqnarray*}
A_1 &= &\{\text{number of loop excursions from $\dint\ball(x_1,L)$ to $\dint\ball(x_1,2L)$ is at least $N$}\},\\
A_2 &= &\{\text{number of loop excursions from $\dint\ball(x_2,L)$ to $\dint\ball(x_2,2L)$ is at least $c_RN$}\},
\end{eqnarray*}
where $c_R = c\,R^{2(2-d)}$. Indeed, on \cite[page 3182]{Ch15} Chang proves that $\P^\alpha[A_2|A_1] \sim 1$ and $\P^\alpha[A_1]\sim c_\alpha\,\rho^N\,N^{\alpha-1}$ as $N\to\infty$ (unformly in $L$).
As a result, $\P^\alpha[A_1\cap A_2] \gg \P^{\alpha(1+\delta)}[A_1]\,\P^{\alpha(1+\delta)}[A_2] \geq c(N)>0$ as $N\to\infty$ (uniformly in $L$). 

\smallskip

In general, events defined by the range of the loop soup are quite different from those defined by loop excursions, 
so the above argument does not disprove \ppp{} for the loop soup. 
(Mind though that existing proofs of decoupling inequalities for random interlacements (and the one of Theorem~\ref{thm:decoupling}) use decompositions into excursions and 
do apply to events $A_1,A_2$, thus if \ppp{} were true for the loop soup, it would be at least hard to verify.)
However, if $d\geq 5$ and $\alpha>0$ small enough 
then for all large $L$, the event that there are at least $2N$ vertex disjoint paths in the range from $\dint\ball(x,L)$ to $\dint\ball(x,2L)$ (later called crossings) 
is essentially equivalent to the event that there are at least $N$ inner loop excursions from $\dint\ball(x,L)$ to $\dint\ball(x,2L)$ 
and $N$ outer excursions from $\dint\ball(x,2L)$ to $\dint\ball(x,L)$. 
(The argument below works for any $\alpha<\alpha_\sharp$, where $\alpha_\sharp$ is the critical threshold for the finiteness of the expected size of the cluster of the origin, see \cite[(2)]{CS16}.) 
More precisely, using the same ideas as in \cite[Section~5]{CS16} one shows that with high probability as $L\to\infty$, 
each crossing from $\dint\ball(x,L)$ to $\dint\ball(x,2L)$ is built from a chain of at most $C\log L$ loops, 
from which exactly one loop has diameter of order $L$ and all the others are of diameter at most $L^{1-2\epsilon}$. 
This implies that every crossing uses an inner or an outer loop excursion between $\dint\ball(x,L+L^{1-\epsilon})$ and $\dint\ball(x,2L-L^{1-\epsilon})$. 
In dimensions $d\geq 5$ with high probability as $L\to\infty$, each excursion is a chain of small sausages linked through cut points, 
which allows to show that each such excursion contributes to exactly one crossing. 
Thus, if the number of crossings from $\dint\ball(x,L)$ to $\dint\ball(x,2L)$ is at least $2N$ (a fixed large number), then 
with high probability as $L\to\infty$, the number of inner and outer loop excursions between $\dint\ball(x,L+L^{1-\epsilon})$ and $\dint\ball(x,2L-L^{1-\epsilon})$
is at least $2N$. 
Vice versa, if the number of excursions between $\dint\ball(x,L-L^{1-\epsilon})$ and $\dint\ball(x,2L+L^{1-\epsilon})$ is at least $2N$, 
then with high probability as $L\to\infty$, the excursions do not intersect each other in $\ball(x,2L)\setminus\ball(x,L)$, which implies that 
the number of crossings from $\dint\ball(x,L)$ to $\dint\ball(x,2L)$ is at least $2N$. 
Using this correspondence between crossings and loop excursions and the above argument of Chang, it is easy to conclude that 
\ppp{} does not hold for the events $\{$number of crossings in the range from $\dint\ball(x_1,L)$ to $\dint\ball(x_1,2L)$ is at least $2N\}$ 
and $\{$number of crossings in the range from $\dint\ball(x_2,L)$ to $\dint\ball(x_2,2L)$ is at least $2c_RN\}$. 
We leave the details of this argument to the reader. 

Although the above reasoning only serves to disprove \ppp{} for the loop soup $\mathscr L^\alpha$ in dimensions $d\geq 5$ and small $\alpha$, 
it is (together with the result of Chang) a good enough evidence that \ppp{} is not a valid condition to study the loop soup. 
Furthermore, in addition to Remark~\ref{rem:d}(5), the argument demonstrates that condition \d{} is weaker than \ppp{}. 
Since by Theorem~\ref{thm:badseed:proba} condition \ppp{} can be replaced by \d{} in all its known applications, 
it is not that interesting to try proving if \ppp{} fails in the remaining cases. 
\end{remark}

\begin{remark}
It is easy to see that the measures $\P^u$ that satisfy \d{}(a) or \d{}(b) are stochastically monotone, i.e., satisfy \pp{}.
The condition is particularly interesting for $\zeta\in(0,1)$, since in this case $e^{(\log L)^\zeta} = o(L^p)$ for any $p>0$.
Furthermore, if $\zeta>\frac12$, then the error term in \eqref{eq:di:increasing} and \eqref{eq:di:decreasing} can be replaced by 
$C\,\exp\left(-c\,\min\left\{(u'-u)^\beta\,s^\gamma,\,(u'-u)^\rho\,e^{(\log L)^\zeta}\right\}\right)$
with an arbitrary $\rho>0$ (see Remark~\ref{rem:zeta}).
\end{remark}

\subsection{Cascading events}

Let $l_k,r_k,L_k$, $k\geq 0$ be sequences of positive integers such that 
\[
L_k = l_{k-1}\cdot L_{k-1}, \quad k \geq 1.
\]
Consider renormalized lattices
\begin{equation*}
\GG_k = L_k \Z^d = \{L_kx : x\in\Z^d\},\quad k\geq 0,
\end{equation*}
and define 
\begin{equation}\label{def:Lambdaxk}
\Lambda_{x,k} = \GG_{k-1}\cap(x+[0,L_k)^d),\quad k\geq 1,\,x\in\GG_k.
\end{equation}
(Note that $|\Lambda_{x,k}| = (l_{k-1})^d$.)

\medskip

For $L_0\geq 1$ and $x\in\GG_0$, any event $\badseed_x = \badseed_{x,0}\in \sigma(\Psi_y, y \in x + [-L_0, 3L_0)^d)$ 
is called a \emph{seed} event. (For simplicity, we omit from notation the dependence of seed events on $L_0$.)
The family of seed events $(\badseed_x:L_0\geq 1,x\in\GG_0)$ is denoted by $\badseed$. 

For $k\geq 1$ and $x\in \GG_k$, we recursively define the events 
\begin{equation}\label{def:cascading}
\badseed_{x,k}= \bigcup_{\scalebox{.8}{$\begin{array}{c} x_1,x_2\in \Lambda_{x,k}\\ \|x_1-x_2\| > r_{k-1}\, L_{k-1} \end{array}$}}
\badseed_{x_1,k-1} \cap \badseed_{x_2,k-1} ~.\
\end{equation}

\medskip

The main result of this section is the following theorem, which states that the result of \cite[Theorem~4.1]{DRS12} holds 
if the family of probability measures $\P^u$ satisfies assumption \d{}. Its proof is given in Section~\ref{sec:badseed:proof}.
\begin{theorem}\label{thm:badseed:proba}
Let $\theta>1$ such that $(\theta+1)\zeta> 1$ and consider the scales 
\begin{equation}\label{def:scales}
l_0,r_0,L_0\geq 1,\quad l_k = l_0\, 4^{\lfloor k^\theta\rfloor},\quad r_k = r_0\, 2^{\lfloor k^\theta\rfloor},\quad L_k = l_{k-1}L_{k-1},\quad k\geq 1.
\end{equation}
Let $\mathbb P^u$, $u\in(a,b)$, be a family of probability measures on $(\Omega,\mathcal F)$. 
Let $\badseed$ be a family of seed events such that for some $u'\in(a,b)$, 
\begin{equation}\label{eq:decoupling:condition}
\liminf_{L_0\to\infty} \sup_{x\in \GG_0} \P^{u'}\left[\badseed_x\right] = 0.
\end{equation}
\begin{itemize}
 \item[(a)]
 If all $\badseed_x$ are increasing and the family $\P^u$ satisfies \d{}(a), 
 then for any $u\in(a,u')$, there exists $C=C(u,u')$ such that 
 for all $l_0\geq 1$, $r_0\geq C(1+\log l_0)^{\frac{2}{\gamma}}$ and some $L_0\geq 1$, 
\begin{equation}\label{eq:decoupling:result}
\sup_{x\in\GG_k} \mathbb P^u\left[\badseed_{x,k}\right] \leq 2^{-2^k} ,\qquad k\geq 0.\
\end{equation}
\item[(b)]
 If all $\badseed_x$ are decreasing and the family $\P^u$ satisfies \d{}(b), 
 then for any $u\in(u',b)$, there exists $C=C(u,u')$ such that 
 for all $l_0\geq 1$, $r_0\geq C(1+\log l_0)^{\frac{2}{\gamma}}$ and some $L_0\geq 1$, \eqref{eq:decoupling:result} holds.
\end{itemize}
Furthermore, if the limit (as $L_0\to\infty$) in \eqref{eq:decoupling:condition} exists (and equals $0$), then 
there exists $C'(u,u',l_0,\badseed)$ such that the statements (a) and (b) hold \emph{for all} $L_0\geq C'$.
\end{theorem}

To study geometric properties of the unique infinite percolation cluster $\set_\infty$ as in \cite{DRS12,PRS,S14}, 
one needs to impose further conditions on the scales $l_k,r_k$, namely, that for all $k\geq 0$, 
$r_k$ divides $l_k$, $l_k>16r_k$ and $\sum_{k=0}^\infty\frac{r_k}{l_k}$ is sufficiently small, see, e.g., below \cite[(37)]{S14}. 
This can be easily achieved, for instance, by taking in \eqref{def:scales} $l_0=r_0^2$ and $r_0$ large enough. 
We briefly summarize the main consequences of Theorem~\ref{thm:badseed:proba}:
\begin{corollary}\label{cor:P1-S2:D}
Assume that a family of probability measures $\P^u$, $u\in(a,b)$, satisfies assumptions \p{}, \pp{}, \d{}, \s{}, \sss{}. 
Then all the results on geometry of $\set_\infty$ from \cite{DRS12,PRS,S14} hold for all $u\in(a,b)$, more precisely, 
\begin{itemize}\itemsep0pt
\item
Theorems~2.3 (chemical distances) and 2.5 (shape theorem) in \cite{DRS12},
\item
Theorem~1.1 in \cite{PRS} (quenched invariance principle),
\item
Theorem~1.13 (Barlow's ball regularity), Corollary~1.14 (quenched Gaussian heat kernel bounds, elliptic and parabolic Harnack inequalities), 
Theorem~1.19 (quenched local CLT), as well as Theorems~1.16--1.18, 1.20 in \cite{S14}.
\end{itemize}
We refer the reader to the introduction of \cite{S14} for the precise statements of these results and relevant discussion.
\end{corollary}

\begin{remark}\label{rem:DRS-conditions-mathcalV}
By Remark~\ref{rem:d}(4), the vacant set $\mathcal V^\alpha$ of random walk loop soup satisfies condition \d{} for all $\alpha>0$. 
Theorem~\ref{thm:lu} proves that $\mathcal V^\alpha$ satisfies condition \s{} for small enough positive $\alpha$. 
(It is believed that \s{} holds for all $\alpha<\alpha_*$, see text below Theorem~\ref{thm:geometry-infinite-cluster}.)
Condition \p{} holds for $\mathcal V^\alpha$ due to \cite[Proposition~3.2]{CS16}. 
Condition \pp{} follows from \d{}, but also directly follows from the definition of $\mathcal V^\alpha$. 
Condition \sss{} holds for $\mathcal V^\alpha$ for all $\alpha<\alpha_*$ by standard arguments of van den Berg and Keane \cite{BK84} --- 
the probability that $0$ is in an infinite cluster of $\mathcal V^\alpha$ is left-continuous for all $\alpha$, since 
it can be expressed as a decreasing limit of non-increasing continuous functions, 
and it is right-continuous for all $\alpha<\alpha_*$, by the uniqueness of the infinite cluster of $\mathcal V^\alpha$, 
see also \cite[Corollary~1.2]{Teixeira:uniqueness}, where the argument of van den Berg and Keane is adapted to 
the vacant set of random interlacements.
(Although the infinite cluster of $\mathcal V^\alpha$ is unique for all $\alpha<\alpha_*$ by an adaptation of the classical Burton-Keane argument, 
see Remark~\ref{rem:uniqueness}, the uniqueness is immediate for $\alpha$ that satisfy \s{} by the Borel-Cantelli lemma.)
Thus, the conclusions of Corollary~\ref{cor:P1-S2:D} hold for $\mathcal V^\alpha$, 
which is the statement of Theorem~\ref{thm:geometry-infinite-cluster}.
\end{remark}

\subsection{Proof of Theorem~\ref{thm:badseed:proba}}\label{sec:badseed:proof}

The proofs of (a) and (b) are essentially the same, we only prove (a). 

Let $\badseed_x$, $x\in\GG_0$ be increasing events and the family $\P^u$ satisfy \d{}(a).  
We assume further that for some $u'\in(a,b)$,
\begin{equation}\label{eq:decoupling:limit}
\lim_{L_0\to\infty} \sup_{x\in \GG_0} \P^{u'}\left[\badseed_x\right] = 0
\end{equation}
and prove that for any $u\in(a,u')$, there exist $C=C(u,u')$ and $C'=C'(u,u',l_0,\badseed)$, such that 
\eqref{eq:decoupling:result} holds for all $l_0\geq 1$, $r_0\geq C(1+\log l_0)^{\frac{2}{\gamma}}$ and $L_0\geq C'$.
It will be seen from the proof how (a) follows if \eqref{eq:decoupling:limit} is replaced by \eqref{eq:decoupling:condition}, 
see the note below \eqref{eq:seedEst}.

\medskip

Let $u\in(a,u')$. Fix $\beta,\gamma,\zeta>0$, for which \d{}(a) holds and define $\chi = \frac{\gamma}{2\beta}>0$ and $\xi = \frac{\gamma}{2}$.

By the choice of $r_k$ in \eqref{def:scales}, there exists $C_1=C_1(u,u')$ such that for all $r_0\geq C_1$, 
\begin{equation}\label{eq:condr0}
\sum_{k=0}^\infty r_k^{-\chi} \leq u'-u .\
\end{equation}
Let 
\begin{equation}\label{eq:uk}
u_0 = u',\qquad u_{k+1} = u_k - r_k^{-\chi},\quad k\geq 0 .\
\end{equation}
By \eqref{eq:condr0}, $u_k\geq u$ for all $k\geq 0$.

\smallskip

Consider the sequence
\begin{equation}\label{eq:Delta}
\Delta_0 = 1 + \sum_{i=0}^\infty\,\frac{\log_2(2l_i^{2d})}{2^{i+1}},\qquad \Delta_{k+1} = \Delta_k - \frac{\log_2(2 l_k^{2d})}{2^{k+1}}, \quad k\geq 0.
\end{equation}
Note that $\Delta_k\geq 1$ for all $k\geq 0$. Since the events $\badseed_x$, $x\in\GG_0$, are increasing, 
the events $\badseed_{x,k}$, $x\in\GG_k$, are also increasing for all $k\geq 0$. Thus, to prove \eqref{eq:decoupling:result} 
it suffices to show that 
\begin{equation}\label{eq:indStatement}
\sup_{x\in\GG_k} \mathbb P^{u_k}\left[\badseed_{x,k}\right] \leq 2^{-\Delta_k\,2^k}, \quad k\geq 0.
\end{equation}
We prove \eqref{eq:indStatement} by induction on $k$. 

\smallskip

\emph{Base of induction:} By the definition of $l_k$ in \eqref{def:scales}, $\Delta_0=\Delta_0(l_0)$. 
Thus, if \eqref{eq:decoupling:limit} holds, then for any $l_0$, there exists $C_1'=C_1'(u,u',l_0,\badseed)$ such that 
\begin{equation} \label{eq:seedEst}
\sup_{x\in\GG_0} \P^{u_0}\left[\badseed_{x,0}\right] \leq 2^{-\Delta_0}
\end{equation}
holds for all $L_0\geq C_1'$. (If only the weaker \eqref{eq:decoupling:condition} is assumed, then 
the existence of (arbitrarily large) $L_0$ for which \eqref{eq:seedEst} holds follows.)

\medskip

\emph{Induction step:} Assume that \eqref{eq:indStatement} holds for some $k\geq 0$ and prove that it also holds for $k+1$. 
Here we will use the definition of events $\badseed_{x,k+1}$ and the assumption \d{}(a). 
Recall that for all $x\in\GG_0$, $\badseed_x\in\sigma(\Psi_y,~y\in x + [-L_0,3L_0)^d)$. 
Thus, by \eqref{def:cascading}, for all $x\in\GG_k$, $\badseed_{x,k}\in \sigma(\Psi_y,~y\in x + [-L_0,L_k + 2L_0)^d)$; 
furthermore, events $\badseed_{x,k}$ are increasing.
Hence, for each $x\in\GG_{k+1}$, 
\begin{multline}\label{eq:indRHS}\itemsep10pt
\P^{u_{k+1}}\left[\badseed_{x,k+1}\right]
\stackrel{\eqref{def:cascading}}\le 
\sum_{x_1, x_2 \in \Lambda_{x,k+1} \; : \; \| x_1 - x_2 \| >  r_{k}\, L_{k}}
\P^{u_{k+1}} \left[ \badseed_{x_1,k}\cap \badseed_{x_2,k}\right] \\
\stackrel{\eqref{eq:di:increasing}, \eqref{eq:uk}}\le \vert \Lambda_{x,k+1} \vert^2
\left( \sup_{x\in\GG_k} \P^{u_k} \left[\badseed_{x,k}\right]^2 + C\,\exp\left(-c\, \min\left\{r_k^{-\chi\beta}\,r_k^\gamma,\,e^{(\log L_k)^\zeta}\right\}\right) \right)\\
\stackrel{\eqref{def:Lambdaxk},\eqref{eq:indStatement}}\le l_k^{2d} 
\left(2^{- \Delta_k 2^{k+1}} + C\,\exp\left(-c\, \min\left\{r_k^\xi,\,e^{(\log L_k)^\zeta}\right\}\right) \right) .\
\end{multline}
To bound \eqref{eq:indRHS} from  above, note that for some $C$, if 
\begin{equation}\label{eq:rk}
\min\left\{r_k^\xi,\,e^{(\log L_k)^\zeta}\right\}\geq C\,\Delta_0\,2^{k+1}\quad \left(\geq C\,\Delta_k\,2^{k+1}\right),
\end{equation}
then \eqref{eq:indRHS} is bounded from above by 
\[
2\,l_k^{2d}\,2^{- \Delta_k 2^{k+1}}\stackrel{\eqref{eq:Delta}}= 2^{-\Delta_{k+1}\,2^{k+1}}.
\]
By the definition of $r_k$ in \eqref{def:scales} and $\Delta_0$ in \eqref{eq:Delta} and using that $\theta>1$, 
the inequality $r_k^\xi\geq C\,\Delta_0\,2^{k+1}$ holds for all $k\geq 0$ if for some $C_2$, $r_0^\xi\geq C_2\,(1 + \log l_0)$. 
Also, by the definition of $L_k$ in \eqref{def:scales} and this time using that $(\theta + 1)\zeta>1$, 
the inequality $e^{(\log L_k)^\zeta}\geq C\,\Delta_0\,2^{k+1}$ holds for all $k\geq 0$ if $L_0\geq C_2'$ for some $C_2'=C_2'(l_0)$.

\smallskip

Thus, we proved that there exist constants $C=C(u,u')$ and $C'=C'(u,u',l_0,\badseed)$, such that 
\eqref{eq:indStatement} holds for all $l_0\geq 1$, $r_0\geq C(1 + \log l_0)^{\frac{1}{\xi}}$ and $L_0\geq C'$. 
(If \eqref{eq:decoupling:condition} is assumed instead of \eqref{eq:decoupling:limit}, 
then \eqref{eq:indStatement} holds for all $l_0\geq 1$, $r_0\geq C(1 + \log l_0)^{\frac{1}{\xi}}$ and any $L_0\geq C_2'$ for which \eqref{eq:seedEst} holds.)
\qed

\begin{remark}\label{rem:zeta}
If $\zeta>\frac12$, we can choose $\theta>1$ in the statement of Theorem~\ref{thm:badseed:proba} such that $(\theta + 1)\zeta>\theta$. 
In this case, for any given $\rho>0$, the inequality $r_k^{-\rho\chi}\,e^{(\log L_k)^\zeta}\geq C\,\Delta_0\,2^{k+1}$ holds for all $k\geq 0$ if $r_0\geq C$ and $L_0\geq C'(l_0)$. 
From the estimate \eqref{eq:rk} it follows that for such choice of $\zeta$, $\theta$ and $\rho$, Theorem~\ref{thm:badseed:proba} holds 
even if the error terms in \d{} are replaced by 
\[
C\,\exp\left(-c\,\min\left\{(u'-u)^\beta\,s^\gamma,\,(u'-u)^\rho\,e^{(\log L)^\zeta}\right\}\right).
\]
\end{remark}

\end{document}